\documentclass[
reqno]{amsart} 
\usepackage[T1]{fontenc}
\usepackage[english]{babel}
\usepackage{amssymb,bbm,enumerate}
\usepackage{color}
\usepackage[linktocpage=true,colorlinks=true, linkcolor=blue, citecolor=red, urlcolor=green]{hyperref}

\usepackage{xypic}

\usepackage{tikz-cd}

\usepackage{faktor}
\usepackage{mathtools}

\usepackage{comment}

\renewcommand{\phi}{\varphi}
\renewcommand{\ker}{\Ker}
\def\ch{\mathrm{ch}}
\def\cl{\mathrm{cl}}

\def\si{\sigma}

\def\Ga{\Gamma}
\def\la{\lambda}

\def\dd{\partial}

\newcommand{\mc}[1]{\mathcal{#1}}

\newcommand{\mb}[1]{\mathbb{#1}}

\newcommand{\ul}[1]{\,\underline {#1}\,}

\DeclareMathOperator{\Hom}{Hom}
\DeclareMathOperator{\End}{End}

\DeclareMathOperator{\ad}{ad}

\DeclareMathOperator{\Der}{Der}

\DeclareMathOperator{\Ker}{Ker}

\DeclareMathOperator{\sign}{sign}

\DeclareMathOperator{\gr}{gr}

\DeclareMathOperator{\dr}{dr}

\DeclareMathOperator{\sesq}{sesq}

\newcommand{\sym}{\mathop{\rm sym }}

\newcommand{\PV}{\mathop{\rm PV }}

\newcommand{\Har}{\mathop{\rm Har }}
\newcommand{\Hoc}{\mathop{\rm Hoc }}

\makeatletter
\def\smallunderbrace#1{\mathop{\vtop{\m@th\ialign{##\crcr
   $\hfil\displaystyle{#1}\hfil$\crcr
   \noalign{\kern3\p@\nointerlineskip}%
   \tiny\upbracefill\crcr\noalign{\kern3\p@}}}}\limits}

\theoremstyle{plain}
\newtheorem{theorem}{Theorem}[section]
\newtheorem{lemma}[theorem]{Lemma}
\newtheorem{proposition}[theorem]{Proposition}
\newtheorem{corollary}[theorem]{Corollary}

\theoremstyle{definition}
\newtheorem{definition}[theorem]{Definition}
\newtheorem{example}[theorem]{Example}

\theoremstyle{remark}
\newtheorem{remark}[theorem]{Remark}

\numberwithin{equation}{section}

\definecolor{light}{gray}{.9}

\begin{document}

\title{Classical and variational Poisson cohomology}

\author{Bojko Bakalov}
\address{Department of Mathematics, North Carolina State University,
Raleigh, NC 27695, USA}
\email{bojko\_bakalov@ncsu.edu}
\urladdr{https://sites.google.com/a/ncsu.edu/bakalov}

\author{Alberto De Sole}
\address{Dipartimento di Matematica, Sapienza Universit\`a di Roma,
P.le Aldo Moro 2, 00185 Rome, Italy}
\email{desole@mat.uniroma1.it}
\urladdr{www1.mat.uniroma1.it/$\sim$desole}

\author{Reimundo Heluani}
\address{IMPA, Rio de Janeiro, Brasil}
\email{rheluani@gmail.com}
\urladdr{http://w3.impa.br/$\sim$heluani/}

\author{Victor G. Kac}
\address{Department of Mathematics, MIT,
77 Massachusetts Ave., Cambridge, MA 02139, USA}
\email{kac@math.mit.edu}
\urladdr{http://www-math.mit.edu/$\sim$kac/}

\author{Veronica Vignoli}
\address{Dipartimento di Matematica, Sapienza Universit\`a di Roma,
P.le Aldo Moro 2, 00185 Rome, Italy}
\email{vignoli@mat.uniroma1.it}



\begin{abstract}
We prove that, for a Poisson vertex algebra $\mc V$,
the canonical injective homomorphism of the variational cohomology of $\mc V$
to its classical cohomology is an isomorphism, provided that $\mc V$,
viewed as a differential algebra, is an algebra of differential polynomials
in finitely many differential variables.
This theorem is one of the key ingredients in the computation of vertex algebra cohomology.
For its proof, we introduce the sesquilinear Hochschild and Harrison cohomology complexes
and prove a vanishing theorem for the symmetric sesquilinear Harrison cohomology
of the algebra of differential polynomials in finitely many differential variables.
\end{abstract}

\maketitle


\pagestyle{plain}

\tableofcontents

\section{Introduction}\label{Intro}

In the series of papers \cite{BDSHK18,BDSHK19,BDSK19,BDSKV19,BDSK21},
the foundations of cohomology theory of vertex algebras have been developed.
The main tool for the computation of this cohomology is the reduction to the variational Poisson
vertex algebra (PVA) cohomology.
The latter is a well-developed theory with many examples computed explicitly \cite{DSK13,BDSK19}.
Its importance stems from the fact that vanishing of the first variational PVA cohomology
leads to the construction of integrable hierarchies of Hamiltonian PDEs.

The reduction of the computation of the vertex algebra cohomology
to the variational PVA cohomology is performed via the classical PVA cohomology
in three steps as follows.
First, let $V$ be a vertex algebra over a field $\mb F$, 
with an increasing filtration by $\mb F[\partial]$-submodules
such that $\mc V:=\gr  V$ carries a canonical structure of a PVA.
Let $(C_{\ch}(V),d)$ be the vertex algebra cohomology complex of $V$.
A filtration on $V$ induces a decreasing filtration on $C_{\ch}(V)$,
and we have a canonical injective map \cite{BDSHK18}:
\begin{equation}\label{eq:1.1}
\gr C_{\ch}(V)\hookrightarrow C_{\cl}(\mc V),
\end{equation}
where $(C_{\cl}(\mc V),\gr d)$ is the classical PVA cohomology complex of $\mc V$.
Moreover, the map \eqref{eq:1.1} is an isomorphism, provided that $\mc V\simeq V$,
as $\mb F[\partial]$-modules \cite{BDSHK19}.

Second, in \cite{BDSK21},
we constructed a spectral sequence from the classical PVA cohomology of $\mc V$
to the vertex algebra cohomology of $V$.

Third, in \cite{BDSHK18}, we constructed a canonical injective map 
\begin{equation}\label{eq:1.2}
H_{\PV}(\mc V)
\hookrightarrow H_{\cl}(\mc V)
\end{equation}
from the variational PVA cohomology of $\mc V$
to its classical PVA cohomology,
and we conjectured that \eqref{eq:1.2} is an isomorphism,
provided that $\mc V$, viewed as a differential algebra,
is an algebra of differential polynomials in finitely many indeterminates.
The main goal of the present paper is to prove this conjecture.

Recall that a Poisson vertex algebra (abbreviated PVA)
is a differential algebra $\mc V$ with a derivation $\partial$,
endowed with a bilinear $\lambda$-bracket $\mc V\times\mc V\to\mc V[\lambda]$,
satisfying the axioms of a Lie conformal algebra and the Leibniz rules \
(see (i)--(iii) and (iv)-(iv'), respectively, in Definition \ref{def:pva}).
In order to construct the variational PVA cohomology complex $(C_{\PV}(\mc V),d)$,
introduce the vector spaces
\begin{equation}\label{eq:1.3}
\mc V_n=\mc V[\lambda_1,\dots,\lambda_n]/
(\partial+\lambda_1+\dots+\lambda_n)\mc V[\lambda_1,\dots,\lambda_n]
\,,\quad
n\geq0
\,,
\end{equation}
where $\lambda_1,\dots,\lambda_n$ are indeterminates.
Then the space of $n$-cochains $C^n_{\PV}(\mc V)$
consists of all linear maps
\begin{equation}\label{eq:1.4}
f\colon\mc V^{\otimes n}\to\mc V_n
\,,\quad
v\mapsto f_{\lambda_1,\dots,\lambda_n}(v),
\end{equation}
satisfying the sesquilinearity conditions \eqref{eq:sesq}, the skewsymmetry conditions \eqref{eq:skew}, and the Leibniz rules \eqref{eq:leib}.
The variational PVA differential $d\colon C^n_{\PV}(\mc V)\to C^{n+1}_{\PV}(\mc V)$ is defined 
by formula \eqref{eq:lca-d}.

In order to define the classical PVA cohomology complex $(C_{\cl}(\mc V),d)$, denote by $\mc G(n)$ the set of oriented graphs with vertices $\{1,\dots,n\}$ and without tadpoles. Then the space of $n$-cochains $C^n_{\cl}(\mc V)$ consists of linear maps (cf. \eqref{eq:1.3}, \eqref{eq:1.4})
\begin{equation}\label{eq:1.5}
Y\colon \mb F\mc G(n)\otimes\mc V^{\otimes n}\to\mc V_n
\,,\quad
\Gamma\otimes v\mapsto Y^{\Gamma}_{\lambda_1,\dots,\lambda_n}(v)
\,,
\end{equation}
satisfying the skewsymmetry conditions \eqref{eq:actclassicoperad}, the cycle relations \eqref{eq:cycle1}, and the sesquilinearity conditions \eqref{eq:sesq2}.
The classical PVA differential is defined by formula \eqref{eq:explicitform}.

The complexes $(C_{\PV}(\mc V),d)$ and $(C_{\cl}(\mc V),d)$ both look similar to the Chevalley--Eilenberg complex for a Lie algebra with coefficients in the adjoint representation. The reason for this similarity is the operadic origin for all these cohomology theories, as explained in \cite{BDSHK18}.

An important observation is that we have a canonical injective map of complexes $\varphi\colon C_{\PV}(\mc V)\to C_{\cl}(\mc V)$ defined by 
\begin{equation}\label{eq:1.6}
\varphi(f)(\Gamma\otimes(v_1\otimes\dots\otimes v_n))
=
\delta_{\Gamma,[n]}f(v_1\otimes\dots\otimes v_n)
\,,
\end{equation}
where $[n]$ denotes the graph with $n$ vertices and no edges.
It was proved in \cite{BDSHK18} that the map \eqref{eq:1.6} induces an injective map in cohomology
\begin{equation}\label{eq:1.7}
\varphi^*\colon H_{\PV}(\mc V)\hookrightarrow H_{\cl}(\mc V)
\,.
\end{equation}
The main result of the present paper is the following (see Theorem \ref{thm:main}).
\begin{theorem}\label{thm:1.1}
Provided that, as a differential algebra, the PVA $\mc V$ is a finitely-generated algebra of differential polynomials, the map $\varphi^*$ is an isomorphism.
\end{theorem}

The proof of this theorem uses the $s$-sesquilinear Hochschild cohomology complex, defined for an associative algebra $A$ with a derivation $\partial$ and a differential bimodule $M$ over $A$ as follows.
For $s=1$, this complex is the differential Hochschild cohomology complex,
for which the space of $n$-cochains is $\Hom_{\mb F[\partial]}(A^{\otimes n},M)$
and the differential $d$ is defined by the usual Hochschild's formula
\begin{equation}\label{eq:1.8}
\begin{split}
(df)(a_1&\otimes\dots\otimes a_{n+1})
=
a_1f(a_2\otimes\dots\otimes a_{n+1}) \\
& +
\sum_{i=1}^n(-1)^i
f(a_1\otimes\dots\otimes a_{i-1}\otimes a_ia_{i+1}\otimes a_{i+2}\otimes\dots\otimes a_{n+1}) \\
& +(-1)^{n+1}
f(a_1\otimes\dots\otimes a_n)a_{n+1}
\,.
\end{split}
\end{equation}
For an arbitrary positive integer $s$, the definition is similar but more complicated.
Given $\ul k=(k_1,\dots,k_s)\,\in\mb Z_{\geq0}^s$, let 
$$
K_0=0
\,,\,\,\,\,
K_t=k_1+\dots+k_t\,, \quad t=1,\dots,s
\,,
$$
and
$$
n=K_s=k_1+\dots+k_s
\,.
$$
Given $v_1,\dots,v_n\in A$,
we denote 
$$
v_{\ul k}^t=v_{K_{t-1}+1}\otimes\dots\otimes v_{K_t}\in A^{\otimes k_t}
\,,\,\, t=1,\dots,s\,,
$$
so that 
$$
v
=
v_1\otimes\dots\otimes v_n
=
v_{\ul k}^1 \otimes\dots\otimes v_{\ul k}^s
\,.
$$
Then the space of $s$-sesquilinear Hochschild $n$-cochains consists of linear maps
(cf. \eqref{eq:1.3}, \eqref{eq:1.4}):
$$
F_{\Lambda_1,\dots,\Lambda_s}\colon
A^{\otimes n}\to M[\Lambda_1,\dots,\Lambda_s]/(\partial+\Lambda_1+\dots+\Lambda_s)M[\Lambda_1,\dots,\Lambda_s]
\,,
$$
satisfying the sesquilinearity conditions ($t=1,\dots,s$),
\begin{equation}\label{eq:1.9}
F_{\Lambda_1,\dots,\Lambda_s}(v_{\ul k}^1\otimes\cdots \partial v_{\ul k}^t\cdots\otimes v_{\ul k}^s)
=
-\Lambda_t
F_{\Lambda_1,\dots,\Lambda_s}(v)
\,.
\end{equation}
The definition of the differential is similar to \eqref{eq:1.8}:
see formulas \eqref{20200623:eq4} and \eqref{20200625:eq2}.
Note that for $s=1$ this coincides with the differential Hochschild complex
if we identify $M$ with $M[\Lambda_1]/(\partial+\Lambda_1)M[\Lambda_1]$.

If $A$ is a commutative associative algebra and $M$ is a symmetric bimodule over $A$,
the differential Hochschild complex contains the Harrison subcomplex,
defined by the Harrison conditions \eqref{eq:harrisoncond}.
We define a similar $s$-sesquilinear Harrison subcomplex of the $s$-sesquilinear
Hochschild complex by Proposition \ref{20200922:prop}.
Moreover, we define by \eqref{20200625:eq3} the action of the symmetric group $S_s$
on the $s$-sesquiinear Harrison complex,
and the symmetric $s$-sesquiinear Harrison complex of $S_s$-invariants,
which we denote by $(C^s_{\sym,\Har}(A,M),d)$.

Our key observation is that the classical PVA complex $(C_{\cl}(\mc V),d)$ is closely related 
to the complex $(C^s_{\sym,\Har}(\mc V, \mc V),d)$.
Namely, introduce an increasing filtration of $C^n_{\cl}$ by letting
$$
F_sC_{\cl}^n
=
\big\{Y\in C_{\cl}^n
\,\big|\,
Y^{\Gamma}=0 \text{ if } s > n-e(\Gamma)
\big\}
\,,
$$
where $e(\Gamma)$ is the number of edges of the graph $\Gamma$.
We prove the following (see Theorem \ref{thm:whatever}):
\begin{theorem}\label{thm:1.2}
For a PVA $\mc V$ and $s\geq1$, we have  a canonical isomorphism of complexes:
$$
\gr_sC_{\cl}(\mc V)
\simeq
C_{\sym,\Har}^s(\mc V,\mc V)
\,,
$$
where on the right the first $\mc V$ is viewed as a commutative associative differential algebra
and the second $\mc V$ as a symmetric bimodule over it.
\end{theorem}

Consequently, Theorem \ref{thm:1.1} follows from Theorem 1.2
and the following vanishing theorem for the sesquilinear Harrison cohomology
(see Theorem \ref{thm:reimundo}).
\begin{theorem}\label{thm:1.3}
Let $\mc V$ be a finitely-generated algebra of differential polynomials.
Then 
$$
H^n(C^s_{\sym,\Har}(\mc V,\mc V),d)
=0 
\quad
\text{ for }
\quad
1\leq s<n
\,.
$$
\end{theorem}

In order to simplify the exposition, we restricted to the purely even case. However, 
the same proofs work in the super case.
Namely, Theorem \ref{thm:1.2} holds for any Poisson vertex superalgebra $\mc V$,
while Theorems \ref{thm:1.1} and \ref{thm:1.3}
hold if $\mc V$ is a superalgebra of differential polynomials
in finitely many commuting and anticommuting indeterminates.

Throughout the paper, the base field $\mb{F}$ has characteristic 0, and, unless
otherwise specified, all vector spaces, their tensor products and Homs are over $\mb{F}$.

\subsubsection*{Acknowledgments} 

This research was partially conducted during the authors' visits
to the University of Rome La Sapienza, to MIT, and to IHES.
The first author was supported in part by a Simons Foundation grant 584741.
The second author was partially supported by the national PRIN fund n. 2015ZWST2C$\_$001
and the University funds n. RM116154CB35DFD3 and RM11715C7FB74D63.
The third author was partially supported by CPNq grant 409582/2016-0.
The fourth author was partially supported by 
the Bert and Ann Kostant fund and by a Simons Fellowship.
We would like to thank Pavel Etingof for providing a proof of the differential
HKR theorem for the algebra of differential polynomials. 

\section{Variational PVA cohomology}\label{sec:2}

\subsection{Poisson vertex algebras}\label{sec:2.1}

\begin{definition}\label{def:pva}
A \emph{Poisson vertex algebra} (PVA) is a differential algebra $\mc V$,
i.e. a commutative associative unital algebra with a derivation $\partial$,
endowed with a bilinear (over $\mb F$) $\lambda$-bracket 
$[\cdot\,_\lambda\,\cdot]\colon \mc V\times \mc V\to \mc V[\lambda]$
satisfying:
\begin{enumerate}[(i)]
\item
sesquilinearity:
$[\partial a_\lambda b]=-\lambda[a_\lambda b]$,
$[a_\lambda \partial b]=(\lambda+\partial)[a_\lambda b]$;
\item
skewsymmetry:
$[a_\lambda b]=-[b_{-\lambda-\partial}a]$,
where $\partial$ is moved to the left to act on coefficients;
\item
Jacobi identity:
$[a_{\lambda}[b_\mu c]]-[b_\mu [a_\lambda ,b]]=[[a_\lambda b]_{\lambda+\mu}c]$.
\item
left Leibniz rule
$[a_{\lambda} bc] = [a_\lambda b]c + [a_\lambda c]b$.
\end{enumerate}
\end{definition}

From the skewsymmetry (ii) and left Leibniz rule (iv) we immediately get the 
\begin{enumerate}[(i')]
\setcounter{enumi}{3}
\item
right Leibniz rule
$[ab_{\lambda} c] = [a_{\lambda+\partial} c]_\to b + [b_{\lambda+\partial} c]_\to a$,
\end{enumerate}
where the arrow means that $\partial$ is moved to the right, acting on $b$ in the first term,
and on $a$ in the second term.

\subsection{Variational PVA complex}\label{sec:2.2}

Given a Poisson vertex algebra $\mc V$, the corresponding 
\emph{variational PVA cohomology complex} $(C_{\PV},d)$
is constructed as follows \cite{DSK13}; see also \cite{BDSK19}.
The space $C_{\PV}^n$ of $n$-cochains consists of linear maps
\begin{equation}\label{eq:lca-maps}
f\colon
\mc V^{\otimes n}
\longrightarrow
\mc V[\lambda_1,\dots,\lambda_n]/\langle\partial+\lambda_1+\cdots+\lambda_n\rangle
\,,
\end{equation}
where $\langle\Phi\rangle$ denotes the image of the endomorphism $\Phi$,
satisfying the \emph{sesquilinearity} conditions ($1\leq i\leq n$):
\begin{equation}\label{eq:sesq}
f_{\lambda_1,\dots,\lambda_n}(v_1\otimes\cdots\otimes (\partial v_i) 
\otimes\cdots\otimes v_n)
=
-\lambda_if_{\lambda_1,\dots,\lambda_n}(v_1\otimes\cdots\otimes v_n)
\,,
\end{equation}
the \emph{skewsymmetry} conditions ($1\leq i< n$):
\begin{equation}\label{eq:skew}
\begin{split}
& f_{\lambda_1,\dots,\lambda_i,\lambda_{i+1},\dots,\lambda_n}
(v_1\otimes\cdots\otimes v_i\otimes v_{i+1} \otimes\cdots\otimes v_n) \\
& =
- f_{\lambda_1,\dots,\lambda_{i+1},\lambda_i,\dots,\lambda_n}
(v_1\otimes\cdots\otimes v_{i+1}\otimes v_i \otimes\cdots\otimes v_n)
\,,
\end{split}
\end{equation}
and the \emph{Leibniz rules} ($1\leq i\leq n$):
\begin{equation}\label{eq:leib}
\begin{split}
\vphantom{\Big(} 
f_{\lambda_1,\dots,\lambda_n}
&(v_1,\dots,u_iw_i,\dots,v_n)
=
f_{\lambda_1,\dots,\lambda_i+\partial,\dots,\lambda_n}(v_1,\dots,u_i,\dots,v_n)_\to w_i \\
& \vphantom{\Big(}
+
f_{\lambda_1,\dots,\lambda_i+\partial,\dots,\lambda_n}(v_1,\dots,w_i,\dots,v_n)_\to u_i
\,.
\end{split}
\end{equation}
For example, 
$C_{\PV}^0=\mc V/\partial \mc V$ and 
$C_{\PV}^1=\Der^{\partial}(\mc V)$ is the space of all derivations of $\mc V$ commuting with $\partial$.

The variational PVA differential
$d\colon C_{\PV}^n\to C_{\PV}^{n+1}$, for $n\geq0$,
is defined by 
\begin{equation}\label{eq:lca-d}
\begin{split}
(df&)_{\lambda_1,\dots,\lambda_{n+1}}
(v_1\otimes\cdots\otimes v_{n+1}) 
=
(-1)^{n}
\sum_{i=1}^{n+1}
(-1)^{i}\,
\big[
{v_i}_{\lambda_i}
f_{\lambda_1,\stackrel{i}{\check{\dots}},\lambda_{n+1}}
(v_1\otimes\stackrel{i}{\check{\dots}}\otimes v_{n+1}) 
\big]
\\
& +
(-1)^{n+1}
\sum_{1\leq i<j\leq n+1}
\!\!\!
(-1)^{i+j}\,
f_{\lambda_i+\lambda_j,\lambda_1,
\stackrel{i}{\check{\dots}}\stackrel{j}{\check{\dots}},\lambda_{n+1}}
([{v_i}_{\lambda_i}v_j]\otimes v_1\otimes
\stackrel{i}{\check{\dots}}\stackrel{j}{\check{\dots}}
\otimes v_{n+1}) 
\,.
\end{split}
\end{equation}
One shows that $d^2=0$, hence we can define the \emph{variational PVA cohomology}
\begin{equation}\label{eq:h-pva}
H_{\PV}(\mc V)
=
\bigoplus_{n\geq0}
H_{\PV}^n(\mc V)
\,\,,\,\,\,\,
H_{\PV}^n(\mc V)
=
\ker\big(d|_{C_{\PV}^{n}}\big)
/
d\big(C_{\PV}^{n-1}\big)
\,.
\end{equation}
\begin{remark}\label{rem:Wnotation}
It was shown in \cite{DSK13} and \cite{BDSHK18},
that the variational PVA cohomology complex associated to the PVA $\mc V$ 
has the structure of a $\mb Z$-graded Lie superalgebra.
The element $X\in
C^2_{\PV}$,
given by
 \begin{equation}\label{eq:PVAstructure1}
X_{\lambda,-\lambda-\partial}(a\otimes b)
=
[a_\lambda b]
\,,
\end{equation}
is odd and satisfies $[X,X] = 0$.
Hence, $(\ad X)^2 = 0$, and $d = \ad X$ was taken as the
differential of the variational PVA cohomology complex. 
As a consequence, the variational PVA cohomology $H_{\PV}(\mc V)$
has an induced Lie superalgebra structure.
Actually, 
what we call here variational PVA cohomology
was called in \cite{DSK13} PVA cohomology;
the variational PVA cohomology was a subcomplex there,
which is equal to the PVA cohomology if $\mc V$ is an algebra of differential polynomials. 
\end{remark}

\section{Preliminaries on the symmetric group and on graphs}\label{sec:3a}

\subsection{Shuffles}\label{sec:shuffles}

A permutation $\sigma \in S_{m+n}$ is called an $(m,n)$-\emph{shuffle} if 
$$
\si(1)< \dots < \si(m)\,,\quad \si(m+1)< \dots <\si(m+n)\,.
$$
The subset of $(m,n)$-shuffles is denoted by $S_{m,n}\subset S_{m+n}$. 
Observe that, by definition, $S_{0,n}=S_{n,0}=\{1\}$ for every $n\geq 0$. If either $m$ or $n$ is negative, 
we set $S_{m,n}= \emptyset$ by convention. 

%
%

\subsection{Monotone permutations} \label{ssec:monotone}

The following notion is due to Harrison \cite{Har62} (see also \cite{GS87}),
and it will be used in Section \ref{sec:5} to define Harrison cohomology.
\begin{definition}
A permutation $\pi \in S_n$ is called \emph{monotone} if, for each $i=1,\dots , n$, 
one of the following two conditions holds:
\begin{enumerate}[(a)]
\item 
$\pi(j)<\pi(i)$ for all $j<i$;
\item
$\pi(j)>\pi(i)$ for all $j<i$.
\end{enumerate}
(Not necessarily the same condition (a) or (b) holds for every $i$.) When (b) holds, 
we call $i$ a \emph{drop} of $\pi$. Also, $\pi(1)=k$ is called the \emph{start} of $\pi$
(and we say that $\pi$ \emph{starts} at $k$).
\end{definition}

We denote by $\mc M_n\subset S_n$ the set of monotone permutations, 
and by $\mc M_n^k\subset \mc M_n$  the set of monotone permutations starting at $k$.

Here is a simple description of all monotone permutations starting at $k$. 
Let us identify the permutation $\pi\in S_n$ with the $n$-tuple $[\pi(1),\dots,\pi(n)]$. 
To construct all $\pi\in\mc M_n^k$, we let $\pi(1)=k$.
Then, for every choice of $k-1$ positions in $\{2,\dots,n\}$ 
we get a monotone permutation $\pi$ as follows. 
In the selected positions we put the numbers $1$ to $k-1$ 
in decreasing order from left to right; in the remaining positions we write the numbers $k+1$ to $n$ 
in increasing order from left to right.
(The selected positions are the drops of $\pi$.)
\begin{example}
The only monotone permutation starting at $1$ is the identity, 
while the only monotone permutation starting at $n$ is 
\begin{equation}\label{eq:sigman}
\sigma_n=[n\;\;n-1\;\cdots\;2\;\;1]\,.
\end{equation}
\end{example}
\begin{example}
Let $n=5$ and $k=3$. The monotone permutations starting at $3$ are 
\begin{align*}
& [3\;\;\ul{2}\;\;\ul{1}\;\;4\;\;5]\,,\quad
[3\;\;\ul{2}\;\;4\;\;\ul{1}\;\;5]\,,\quad
[3\;\;\ul{2}\;\;4\;\;5\;\;\ul{1}] \,,\\
& [3\;\;4\;\;\ul{2}\;\;\ul{1}\;\;5] \,,\quad
[3\;\;4\;\;\ul{2}\;\;5\;\;\ul{1}] \,,\quad
[3\;\;4\;\;5\;\;\ul{2}\;\;\ul{1}] \,,
\end{align*}
where we underlined the positions of the drops.
\end{example}
Given a monotone permutation $\pi$, we denote by 
$\dr(\pi)$ the sum of all the drops with respect to $\pi$. 
According to the previous description, we can easily see that 
\begin{equation} \label{eq:drop}
(-1)^{\dr(\pi)}=(-1)^{k-1} \sign (\pi)\,,
\end{equation} 
if $k$ is the start of $\pi$.

\subsection{Graphs}\label{ssec:ngraphs}

For an oriented graph $\Gamma$, we denoted by $V(\Gamma)$ the set of vertices of $\Ga$, 
and by $E(\Gamma)$ the set of edges. 
We call an oriented graph $\Ga$ an $n$-\emph{graph} 
if $V(\Gamma)=\{1, \dots, n\}$. Denote by $\mc G(n)$ the set of all $n$-graphs without tadpoles, 
and by $\mc G_0(n)$ the set of all acyclic $n$-graphs. 

An $n$-graph $L$ will be called an $n$-\emph{line}, or simply a \emph{line}, 
if its set of edges is of the form $\{i_1\to i_2,\,i_2\to i_3,\dots,\,i_{n-1}\to i_n\}$,
where $\{i_1,\dots,i_n\}$ is a permutation of $\{1,\dots,n\}$.

We have a natural left action of $S_n$ on the set $\mc G(n)$: 
for the $n$-graph $\Ga$ and the permutation $\sigma$, the new $n$-graph $\sigma(\Ga)$ 
is defined to be the  same graph as $\Ga$ but with the vertex which was labeled as $i$ relabeled 
as $\sigma(i)$, for every $i=1,\dots,n$. 
So, if the $n$-graph $\Ga$ has an oriented edge $i \rightarrow j$, 
then the $n$-graph $\sigma(\Ga)$ has the oriented edge $\sigma(i) \rightarrow \sigma(j)$. 
Obviously, $S_n$ permutes the set of $n$-lines.

\begin{example}\label{ex:pictures}
Let
	\begin{equation*}
	\begin{tikzpicture}
	\node at (-1,0) {$\Gamma=$};
	\draw[fill] (0,0) circle [radius=0.1];
	\node at (0,-0.3) {1};
	\draw[fill] (1,0) circle [radius=0.1];
	\node at (1,-0.3) {2};
	\draw[fill] (2,0) circle [radius=0.1];
	\node at (2,-0.3) {3};
	\draw[fill] (3,0) circle [radius=0.1];
	\node at (3,-0.3) {4};
	\draw[fill] (4,0) circle [radius=0.1];
	\node at (4,-0.3) {5};
	\draw[fill] (5,0) circle [radius=0.1];
	\node at (5,-0.3) {6};
	\draw[->] (0.1,0) -- (0.9,0);
	\draw[->] (0,-0.1) to [out=270,in=270] (2,-0.1);
	\draw[->] (1,0.1) to [out=90,in=90] (3,0.1);
	\draw[->] (1,0.1) to [out=90,in=90] (5,0.1);
	\node at (5.3,0) {.};
	\end{tikzpicture}
	\end{equation*}
For $\sigma=(6\;5\;4)$ and $\tau = \left(\begin{array}{cc|cccccccccc}
	1 & 2  & 3 & 4 & 5 & 6  \\
	3 & 4 & 1 & 5 & 6 & 2 
	\end{array} \right)$, 
we have:
	
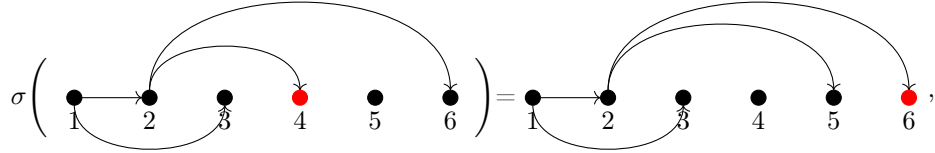
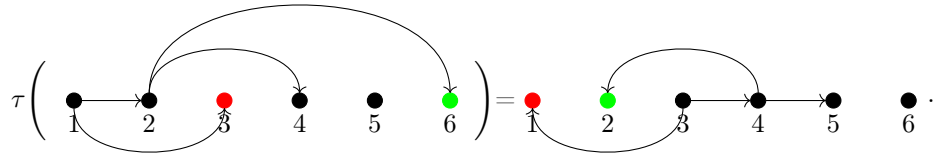
\begin{equation*}
	\begin{tikzpicture}
	\node at (-0.6,0) {$\sigma\Bigg($};
	\draw[fill] (0,0) circle [radius=0.1];
	\node at (0,-0.3) {1};
	\draw[fill] (1,0) circle [radius=0.1];
	\node at (1,-0.3) {2};
	\draw[fill] (2,0) circle [radius=0.1];
	\node at (2,-0.3) {3};
	\draw[fill,red] (3,0) circle [radius=0.1];
	\node at (3,-0.3) {4};
	\draw[fill] (4,0) circle [radius=0.1];
	\node at (4,-0.3) {5};
	\draw[fill] (5,0) circle [radius=0.1];
	\node at (5,-0.3) {6};
	\draw[->] (0.1,0) -- (0.9,0);
	\draw[->] (0,-0.1) to [out=270,in=270] (2,-0.1);
	\draw[->] (1,0.1) to [out=90,in=90] (3,0.1);
	\draw[->] (1,0.1) to [out=90,in=90] (5,0.1);
	\end{tikzpicture}
	\begin{tikzpicture}
	\node at (-0.5,0) {\Bigg)=};
	\draw[fill] (0,0) circle [radius=0.1];
	\node at (0,-0.3) {1};
	\draw[fill] (1,0) circle [radius=0.1];
	\node at (1,-0.3) {2};
	\draw[fill] (2,0) circle [radius=0.1];
	\node at (2,-0.3) {3};
	\draw[fill] (3,0) circle [radius=0.1];
	\node at (3,-0.3) {4};
	\draw[fill] (4,0) circle [radius=0.1];
	\node at (4,-0.3) {5};
	\draw[fill,red] (5,0) circle [radius=0.1];
	\node at (5,-0.3) {6};
	\draw[->] (0.1,0) -- (0.9,0);
	\draw[->] (0,-0.1) to [out=270,in=270] (2,-0.1);
	\draw[->] (1,0.1) to [out=90,in=90] (4,0.1);
	\draw[->] (1,0.1) to [out=90,in=90] (5,0.1);
	\node at (5.3,0) {,};
	\end{tikzpicture}
	\end{equation*}
	and
	\begin{equation*}
	\begin{tikzpicture}
	\node at (-0.6,0) {$\tau\Bigg($};
	\draw[fill] (0,0) circle [radius=0.1];
	\node at (0,-0.3) {1};
	\draw[fill] (1,0) circle [radius=0.1];
	\node at (1,-0.3) {2};
	\draw[fill,red] (2,0) circle [radius=0.1];
	\node at (2,-0.3) {3};
	\draw[fill] (3,0) circle [radius=0.1];
	\node at (3,-0.3) {4};
	\draw[fill] (4,0) circle [radius=0.1];
	\node at (4,-0.3) {5};
	\draw[fill,green] (5,0) circle [radius=0.1];
	\node at (5,-0.3) {6};
	\draw[->] (0.1,0) -- (0.9,0);
	\draw[->] (0,-0.1) to [out=270,in=270] (2,-0.1);
	\draw[->] (1,0.1) to [out=90,in=90] (3,0.1);
	\draw[->] (1,0.1) to [out=90,in=90] (5,0.1);
	\end{tikzpicture}
	\begin{tikzpicture}
	\node at (-0.5,0) {\Bigg)=};
	\draw[fill,red] (0,0) circle [radius=0.1];
	\node at (0,-0.3) {1};
	\draw[fill,green] (1,0) circle [radius=0.1];
	\node at (1,-0.3) {2};
	\draw[fill] (2,0) circle [radius=0.1];
	\node at (2,-0.3) {3};
	\draw[fill] (3,0) circle [radius=0.1];
	\node at (3,-0.3) {4};
	\draw[fill] (4,0) circle [radius=0.1];
	\node at (4,-0.3) {5};
	\draw[fill] (5,0) circle [radius=0.1];
	\node at (5,-0.3) {6};
	\draw[->] (2.1,0) -- (2.9,0);
	\draw[<-] (0,-0.1) to [out=270,in=270] (2,-0.1);
	\draw[->] (3.1,0) -- (3.9,0);
	\draw[<-] (1,0.1) to [out=90,in=90] (3,0.1);
	\node at (5.3,0) {.};
	\end{tikzpicture}
	\end{equation*}
\end{example}

\subsection{Graphs of type $\ul k$ and proper $\ul k$-lines} \label{ssec:lines}

For $s\geq1$, let 
$$
\ul k=(k_1,\dots,k_s)\,\in\mb Z_{\geq0}^s
\,\,\text{ and }\,\,
n=k_1+\dots+k_s
\,,
$$
and denote
\begin{equation} \label{eq:notation}
K_0=0
\,\,\text{ and }\,\,
K_t=k_1+\dots+k_t\,, \quad t=1,\dots,s\,,
\end{equation}
so that $K_s=n$.
We denote by $\Gamma_{\ul k}\in\mc G(n)$ 
the \emph{standard} $\ul k$-\emph{line}, 
union of connected lines of lengths $k_1,\dots,k_s$, 
with the labeling of the vertices ordered from left to right:
\begin{equation}\label{eq:Gammak}
\begin{tikzpicture}
\node at (-0.2,1) {$\Gamma_{\ul{k}}=$};
\draw[fill] (0.5,1) circle [radius=0.07];
\node at (0.5,0.6) {\tiny{$1$}};
\draw[->] (0.6,1) -- (0.9,1);
\draw[fill] (1,1) circle [radius=0.07];
\node at (1,0.6) {\tiny{$2$}};
\draw[->] (1.1,1) -- (1.4,1);
\node at (1.7,1) {$\cdots$};
\draw[->] (1.9,1) -- (2.2,1);
\draw[fill] (2.3,1) circle [radius=0.07];
\node at (2.3,0.6) {\tiny{$K_1$}};
\draw[fill] (3,1) circle [radius=0.07];
\node at (3,0.6) {\tiny{$K_1\!+\!1$}};
\draw[->] (3.1,1) -- (3.4,1);
\draw[fill] (3.5,1) circle [radius=0.07];
\draw[->] (3.6,1) -- (3.9,1);
\node at (4.2,1) {$\cdots$};
\draw[->] (4.4,1) -- (4.7,1);
\draw[fill] (4.8,1) circle [radius=0.07];
\node at (4.8,0.6) {\tiny{$K_2$}};
\node at (6,1) {$\cdots$};
\draw[fill] (7.2,1) circle [radius=0.07];
\node at (7.2,0.6) {\tiny{$K_{s-1}\!+\!1$}};
\draw[->] (7.3,1) -- (7.6,1);
\node at (7.9,1) {$\cdots$};
\draw[->] (8.2,1) -- (8.5,1);
\draw[fill] (8.6,1) circle [radius=0.07];
\node at (8.6,0.6) {\tiny{$n$}};
%
\end{tikzpicture}
\end{equation}
We allow some of the $k_i$'s to be zero,
in which case the corresponding connected component of $\Gamma_{\ul k}$ is empty.
In the special case $s=1$ we recover
the \emph{standard} $n$-\emph{line}
\begin{equation} \label{eq:Ln}
\begin{tikzpicture}[scale=1.5]
\node at (-0.5,0){$\Ga_n=$};
\draw[fill] (0,0) circle [radius=0.02];
\node at (0,-0.2){\tiny{$1$}};
\draw[->] (0,0) to (0.3,0);
\draw[fill] (0.3,0) circle [radius=0.02];
\node at (0.3,-0.2){\tiny{$2$}};
\draw[->] (0.3,0) to (0.6,0);
\node at (0.7,0) {.};
\node at (0.75,0) {.};
\node at (0.8,0) {.};
\draw[->] (0.9,0) to (1.2,0);
\draw[fill] (1.2,0) circle [radius=0.02];
\node at (1.2,-0.2){\tiny{$n$}};
\node at (1.5,0) {.};
\end{tikzpicture}
\end{equation}

An arbitrary $\ul k$-\emph{line} is obtained by permuting the vertices of $\Gamma_{\ul k}$:
\begin{equation}\label{eq:unionlines}
\begin{tikzpicture}
\node at (-0.2,1) {$\Gamma=$};
\draw[fill] (0.5,1) circle [radius=0.07];
\node at (0.5,0.6) {$i^1_1$};
\draw[->] (0.6,1) -- (0.9,1);
\draw[fill] (1,1) circle [radius=0.07];
\node at (1,0.6) {$i^1_2$};
\draw[->] (1.1,1) -- (1.4,1);
\node at (1.7,1) {$\cdots$};
\draw[->] (1.9,1) -- (2.2,1);
\draw[fill] (2.3,1) circle [radius=0.07];
\node at (2.3,0.6) {$i^1_{k_1}$};
\draw[fill] (3,1) circle [radius=0.07];
\node at (3,0.6) {$i^2_1$};
\draw[->] (3.1,1) -- (3.4,1);
\draw[fill] (3.5,1) circle [radius=0.07];
\node at (3.5,0.6) {$i^2_2$};
\draw[->] (3.6,1) -- (3.9,1);
\node at (4.2,1) {$\cdots$};
\draw[->] (4.4,1) -- (4.7,1);
\draw[fill] (4.8,1) circle [radius=0.07];
\node at (4.8,0.6) {$i^2_{k_2}$};
\node at (5.45,1) {$\cdots$};
\draw[fill] (6,1) circle [radius=0.07];
\node at (6,0.6) {$i^s_1$};
\draw[->] (6.1,1) -- (6.4,1);
\draw[fill] (6.5,1) circle [radius=0.07];
\node at (6.5,0.6) {$i^s_2$};
\draw[->] (6.6,1) -- (6.9,1);
\node at (7.2,1) {$\cdots$};
\draw[->] (7.4,1) -- (7.7,1);
\draw[fill] (7.8,1) circle [radius=0.07];
\node at (7.8,0.6) {$i^s_{k_s}$};
%
\end{tikzpicture}
\end{equation}
where the set of indices $\{i^a_b\}$ is a permutation of $\{1,\dots,n\}$.
Note that, if $\Gamma$ is a $\ul k$-line,
then it is a $\sigma(\ul k)$-line for every permutation $\sigma\in S_s$.
Hence, when considering $\ul k$-lines we can (and we will) assume that
$k_1\leq\dots\leq k_s$.
We say that a $\ul k$-line is \emph{proper} if the following further condition
holds on the indices of the vertices:
\begin{equation}\label{eq:min}
i^l_1=\min\{i_1^l,\dots,i_{k_l}^l\} \;\;\;\forall\;l=1,\dots,s\,.
\end{equation}
We then let
\begin{equation}\label{eq:mcLn}
\mc L(n)
\,=\,
\Big\{\text{proper $\ul k$-lines }\, \Gamma\in\mc G(n)\,\text{ with }\, 
\ul k\in\mb Z_{\geq1}^s,\, s\geq1,\, k_1+\dots+k_s=n
\Big\}
\,.
\end{equation}
Note that, in order not to have repetitions in the set \eqref{eq:mcLn},
we may assume that $k_1\leq\dots\leq k_s$,
and that, if $k_l=k_{l+1}$, then $i_1^l<i_1^{l+1}$.
Obviously, $\Gamma_{\ul k}\in\mc L(n)$ for every $\ul k\in\mb Z_{\geq1}^s$,
while a permutation of $\Gamma_{\ul k}$ does not necessarily lie in $\mc L(n)$.

Finally, we say that a graph $\Gamma\in\mc G(n)$ is of \emph{type} $\ul k$
if it is disjoint union of $s$ connected components of sizes $k_1\leq\dots\leq k_s$.
Obviously, any $\ul k$-line is of type $\ul k$.

We can extend the definition of $\Gamma_{\ul k}$ for $\ul k\in\mb Z_{\geq0}^s$
by removing all $0$'s from $\ul k$.
In particular $\Gamma_0$ is the empty graph.

\subsection{Cycle relations on graphs} \label{ssec:cyclerel}

Let $\mb F\mc G (n)$ be the vector space 
with basis the set of graphs $\mc{G}(n)$,
and $\mc R(n)\subset\mb F\mc G(n)$ be the subspace spanned 
by the following \emph{cycle relations}:
\begin{enumerate}[(i)]
\item\label{eq:cyclerel1}
all $\Ga\in \mc G(n)\setminus \mc{G}_{0}(n)$ (i.e., graphs containing a cycle); 
\item\label{eq:cyclerel2}
all linear combinations $\sum_{e\in C} \Ga\setminus e$, where $\Ga\in \mc{G}(n)$ 
and $C\subset E(\Ga)$ is an oriented cycle. 
\end{enumerate}
By convention, $\mb F\mc G(0)=\mb F$ and $\mc R(0)=0$.

Note that reversing an arrow in a graph $\Ga\in \mc{G}(n)$ gives us, modulo cycle relations, 
the element $-\Ga\in \mb{F}\mc{G}(n)$.
For example, for $n=3$, a cycle relation of type \eqref{eq:cyclerel2} is:
\begin{equation}\label{eq:triangle}
\begin{tikzpicture}
\node at (0,0) {$2$};
\node at (1.4,0) {$3$};
\node at (0.7,1) {$1$};
\draw[->] (0.6,0.9) to (0.1,0.1);
\draw[->] (0.1,0) to (1.3,0);
\end{tikzpicture}
\genfrac{}{}{0pt}{}{\;\;+\;\;}{}
\begin{tikzpicture}
\node at (0,0) {$2$};
\node at (1.4,0) {$3$};
\node at (0.7,1) {$1$};
\draw[->] (0.1,0) to (1.3,0);
\draw[<-] (0.8,0.9) to (1.3,0.1);
\end{tikzpicture}
\genfrac{}{}{0pt}{}{\;\;+\;\;}{}
\begin{tikzpicture}
\node at (0,0) {$2$};
\node at (1.4,0) {$3$};
\node at (0.7,1) {$1$};
\draw[<-] (0.8,0.9) to (1.3,0.1);
\draw[->] (0.6,0.9) to (0.1,0.1);
\end{tikzpicture}
\end{equation}
\begin{theorem}[{\cite[Theorem\ 4.7]{BDSHK19}}] \label{thm:lines}
The set\/ $\mathcal{L}(n)$ is a basis for the quotient space\/ $\mb{F}\mc{G}(n)/\mc R(n)$.
\end{theorem}
%

\subsection{Harrison relations} \label{ssec:harrel}

The following result will be used in Section \ref{sec:6}.
\begin{lemma}{\cite[Lemma 4.8]{BDSKV19}}\label{lem:identity}
	Let\/ $\Ga_n$ be the standard $n$-line, as in \eqref{eq:Ln}. 
	For every $m\in \{2,\dots,n\}$, the following identity holds:
	\begin{equation} \label{eq:identity}
	\Ga_n + (-1)^{m} \sum_{\pi\in \mc M_n^m} \pi \Ga_n \in \mc R(n)\,,
	\end{equation} 
	where the sum is over all monotone permutations $\pi$ starting at $m$ and the action of $S_n$
	on graphs is described in Section \ref{ssec:ngraphs}.
\end{lemma}

\subsection{Notation for subgraphs and collapsed graphs} \label{ssec:graphnotation}

Let us introduce the following notation. For $h\in\{1,\ldots,n\}$ and $\Ga\in \mc G(n)$, 
we denote by $\Ga\backslash h\in\mc G(n-1)$ 
the complete subgraph obtained from $\Gamma$ by removing the vertex $h$
and all edges starting or ending in $h$,
and relabeling the vertices from $1$ to $n-1$.
Moreover, for $i,j\in\{1,\dots,n\}$,
we define the graph $\pi_{ij}(\Ga)\in\mc G(n-1)$ obtained by collapsing the vertices $i$ and $j$
(and any edges between them) into a single vertex, numbered by $1$, 
and renumbering the remaining vertices from $2$ to $n-1$.
\begin{example}
For example, if 
$$
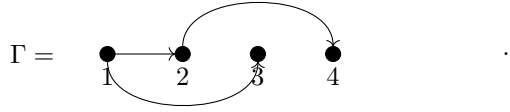
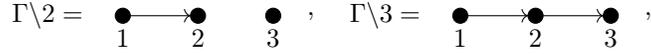
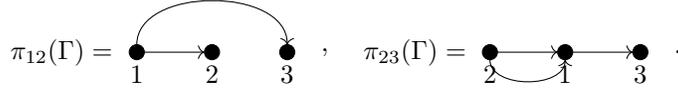
\begin{tikzpicture}
\node at (-1,0) {$\Gamma=$};
\draw[fill] (0,0) circle [radius=0.1];
\node at (0,-0.3) {1};
\draw[fill] (1,0) circle [radius=0.1];
\node at (1,-0.3) {2};
\draw[fill] (2,0) circle [radius=0.1];
\node at (2,-0.3) {3};
\draw[fill] (3,0) circle [radius=0.1];
\node at (3,-0.3) {4};
%
\draw[->] (0.1,0) -- (0.9,0);
\draw[->] (0,-0.1) to [out=270,in=270] (2,-0.1);
\draw[->] (1,0.1) to [out=90,in=90] (3,0.1);
\node at (5.3,0) {.};
\end{tikzpicture}
$$
we have
$$
\begin{tikzpicture}
\node at (-1,0) {$\Ga\backslash2=$};
\draw[fill] (0,0) circle [radius=0.1];
\node at (0,-0.3) {1};
\draw[fill] (1,0) circle [radius=0.1];
\node at (1,-0.3) {2};
\draw[fill] (2,0) circle [radius=0.1];
\node at (2,-0.3) {3};
\draw[->] (0.1,0) -- (0.9,0);
\node at (2.5,0) {,};
\end{tikzpicture}
\,\,\,\,
\begin{tikzpicture}
\node at (-1,0) {$\Ga\backslash3=$};
\draw[fill] (0,0) circle [radius=0.1];
\node at (0,-0.3) {1};
\draw[fill] (1,0) circle [radius=0.1];
\node at (1,-0.3) {2};
\draw[fill] (2,0) circle [radius=0.1];
\node at (2,-0.3) {3};
\draw[->] (0.1,0) -- (0.9,0);
\draw[->] (1.1,0) -- (1.9,0);
\node at (2.5,0) {,};
\end{tikzpicture}
$$
and
$$
\begin{tikzpicture}
\node at (-1,0) {$\pi_{12}(\Ga)=$};
\draw[fill] (0,0) circle [radius=0.1];
\node at (0,-0.3) {1};
\draw[fill] (1,0) circle [radius=0.1];
\node at (1,-0.3) {2};
\draw[fill] (2,0) circle [radius=0.1];
\node at (2,-0.3) {3};
\draw[->] (0.1,0) -- (0.9,0);
\draw[->] (0,0.1) to [out=90,in=90] (2,0.1);
\node at (2.5,0) {,};
\end{tikzpicture}
\,\,\,\,
\begin{tikzpicture}
\node at (-1,0) {$\pi_{23}(\Ga)=$};
\draw[fill] (0,0) circle [radius=0.1];
\node at (0,-0.3) {2};
\draw[fill] (1,0) circle [radius=0.1];
\node at (1,-0.3) {1};
\draw[fill] (2,0) circle [radius=0.1];
\node at (2,-0.3) {3};
\draw[->] (0.1,0) -- (0.9,0);
\draw[->] (1.1,0) -- (1.9,0);
\draw[->] (0,-0.1) to [out=270,in=270] (1,-0.1);
\node at (2.5,0) {.};
\end{tikzpicture}
$$
\end{example}
For $\Ga\in \mc G_0(n)$ and $i\in\{1,\dots,n\}$, we denote by 
$\deg_{\Ga}^-(i)$ the indegree of $i$ in $\Ga$, namely the number of edges of $\Ga$ incoming to $i$, 
by $\deg_{\Ga}^+(i)$ the outdegree of $i$ in $\Ga$, namely the number of edges of $\Ga$ outcoming from $i$,
and 
$$
\deg_{\Ga}(i)\coloneqq \deg_{\Ga}^-(i) + \deg_{\Ga}^+(i)
\,,
$$ 
the degree of $i$ in $\Ga$. 
For $i,j\in\{1,\dots,n\}$, we also let
$$
\epsilon_{\Ga}(i,j)\coloneqq \begin{cases}
1 & \text{ if  }\,\, i\rightarrow j\in E(\Ga)\,, \\
-1 & \text{ if  }\,\, i\leftarrow j\in E(\Ga)\,, \\
0 & \text{ otherwise}\,.
\end{cases}
$$
Note that, since $\Ga\in \mc G_0(n)$, $i\rightarrow j$ and $j\rightarrow i$ cannot be both in $E(\Ga)$.

\section{Classical PVA cohomology}\label{sec:3}

\subsection{Space of classical cochains}\label{sec:3.1}

Let $\mc V$ be a Poisson vertex algebra.
The corresponding \emph{classical PVA cohomology complex} $(C_{\cl},d)$ is constructed as follows \cite{BDSHK18}.
The space $C_{\cl}^{n}$ of \emph{classical $n$-cochains} consists of linear maps
\begin{equation}\label{eq:operadcl}
Y\colon \mb F\mc G(n)\otimes \mc V^{\otimes n} \longrightarrow 
\mc V[\lambda_1,\dots,\lambda_n]\big/\langle\partial+\lambda_1+\dots+\lambda_n\rangle\,,
\end{equation}
mapping the $n$-graph $\Gamma\in\mc G(n)$ 
and the monomial $v_1\otimes\,\cdots\,\otimes v_n\in V^{\otimes n}$ to the polynomial
\begin{equation}\label{eq:imcl}
Y^{\Gamma}_{\lambda_1,\dots,\lambda_n}(v_1\otimes\,\cdots\,\otimes v_n)\,,
\end{equation}
satisfying the skewsymmetry conditions,  cycle relations, and sesquilinearity conditions described below.

The \emph{skewsymmetry conditions} on $Y$ say that, for each permutation $\sigma\in S_n$,
we have 
\begin{equation}\label{eq:actclassicoperad}
Y^{\sigma(\Gamma)}_{\lambda_1,\dots,\lambda_n}
(v_1\otimes\cdots\otimes v_n)
=
\sign(\sigma)
Y^\Gamma_{\lambda_{\sigma(1)},\dots,\lambda_{\sigma(n)}}
(v_{\sigma(1)}\otimes\cdots\otimes v_{\sigma(n)})
\,,  
\end{equation}
where $\sigma(\Gamma)$ is defined in Section \ref{ssec:ngraphs}. 

Recall that $\mc R(n)\subset\mb F\mc G(n)$ is the subspace spanned by the cycle relations (i) and (ii)
from Section \ref{ssec:cyclerel}.
The \emph{cycle relations} on $Y$ say that
\begin{equation}\label{eq:cycle1}
Y^{\Gamma}=0 \,\,\text{ for }\,\, \Gamma \in \mc R(n) 
\,.
\end{equation}
Hence, $Y$ induces a map on $\mb F\mc G(n)/\mc R(n)$.
As an example, observe that, by the first cycle relation (i), 
changing orientation of a single edge of the $n$-graph $\Gamma\in\mc G(n)$ amounts 
to the change of sign of $Y^{\Gamma}$.

Let $\Gamma=\Gamma_1\sqcup\dots\sqcup\Gamma_s$
be the decomposition of $\Gamma$ as a disjoint union of its connected components,
and let $I_1,\dots,I_s\subset\{1,\dots,n\}$ be the sets of vertices of these
connected components. 
For each $\Ga_\alpha$ we write
\begin{equation}\label{eq:notation3}
\lambda_{\Ga_\alpha} 
= \sum_{i\in I_\alpha} \lambda_{i}\,,\qquad \partial_{\Ga_\alpha}
= \sum_{i\in I_\alpha} \partial_i\,,
\end{equation}
where $\partial_i$ denotes the action of $\partial$ on the $i$-th factor in the tensor product $\mc V^{\otimes n}$. 
Then, the \emph{sesquilinearity conditions} on $Y$  say that,
for $v\in\mc V^{\otimes n}$,
\begin{equation}\label{eq:sesq1}
Y^{\Gamma}_{\lambda_1,\dots,\lambda_n}(v)
\,\,\text{ is a polynomial in }\,\,
\lambda_{\Gamma_1},\dots,\lambda_{\Gamma_s}
\,,
\end{equation}
(and not in the variables $\lambda_1,\dots,\lambda_n$ separately),
and, for every $\alpha=1,\dots,s$,
\begin{equation}\label{eq:sesq2}
Y^{\Gamma}_{\lambda_1,\dots,\lambda_n}
((\partial_{\Gamma_\alpha}+\lambda_{\Gamma_\alpha})v)
=0
\,.
\end{equation}
Observe that the second sesquilinearity condition \eqref{eq:sesq2} implies
\begin{equation}\label{eq:sesq2.1}
Y^{\Gamma}_{\lambda_1,\dots,\lambda_n}(\partial v) = 
-\sum_{i=1}^n \la_i \, Y^{\Gamma}_{\lambda_1,\dots,\lambda_n}(v)
= \partial \bigl( Y^{\Gamma}_{\lambda_1,\dots,\lambda_n}(v) \bigr), 
\qquad v\in \mc V^{\otimes n}\,,
\end{equation}
i.e. $Y^\Gamma:\,\mc V^{\otimes n}\to 
\mc V[\lambda_1,\dots,\lambda_n]\big/\langle\partial+\lambda_1+\dots+\lambda_n\rangle$
is an $\mb F[\partial]$-module homomorphism.
\begin{remark}\label{rem:connected}
When the graph $\Gamma$ is connected, the first sesquilinearity condition \eqref{eq:sesq1} implies that
$Y^{\Gamma}_{\lambda_1,\dots,\lambda_n}(v)$
is a polynomial of
$\lambda_1+\dots+\lambda_n\equiv-\partial$. Hence, it is an element of
$$
\mc V[\lambda_1+\dots+\lambda_n]\big/\langle\partial+\lambda_1+\dots+\lambda_n\rangle
\simeq \mc V \,.
$$
In this case, we will omit the subscript of $Y^{\Gamma}$.
\end{remark}

By convention, for $n=0$ the graph $\Gamma$ is empty and $s=0$;
hence $C_{\cl}^{0}=\mc V/\partial\mc V$.
Note also that $C_{\cl}^1=\End_{\mb F[\partial]}\mc V$.

\subsection{Differential}\label{sec:3.2}

The classical PVA cohomology differential $d\colon C_{\cl}^{n}\to C_{\cl}^{n+1}$
is defined by the following formula:
\begin{align} \label{eq:explicitform}
&(dY)^{\Ga}_{\la_1,\ldots,\la_{n+1}}(v_1\otimes\ldots\otimes v_{n+1}) \notag \\
&=
\sum_{h:\, \deg_{\Ga}(h)=0} (-1)^{n-h} \Big[{v_h}_{\la_h} 
Y^{\Ga\backslash h}_{\la_1,\ldots\overset{h}{\curlyvee}\ldots,\la_{n+1}} 
(v_1\otimes\ldots \overset{h}{\curlyvee}\ldots\otimes v_{n+1}) \Big] \notag \\
&+
\!\!\!
\sum_{\substack{h:\, \deg_{\Ga}(h)=1 \\ j:\, \epsilon_{\Ga}(j,h)\neq 0}} 
\!\!\!\!\!\!
(-1)^{\deg_{\Ga}^+(h)+n-h+1} 
Y^{\Ga\backslash h}_{\la_1,\ldots\overset{h}{\curlyvee}\ldots,\la_j+x,\ldots,\la_{n+1}} 
\!\!\!\!\!\!\!\!\!
(v_1\otimes\ldots\overset{h}{\curlyvee}\ldots\otimes v_{n+1}) 
\big(\big|_{x=\la_h+\partial} v_h\big) \notag \\
&+
\sum_{i<j:\,\epsilon_{\Ga}(i,j)=0} 
(-1)^{n+i+j-1} 
Y^{\pi_{ij}(\Ga)}_{\la_i+\la_j,\la_1,\ldots\overset{i,j}{\curlyvee}\ldots,\la_{n+1}} 
\Big(
[{v_i}_{\la_i+X(i)}v_j]\otimes \notag \\
&\,\,\,\,\,\,\,\,\,\,\,\,\,\,\,\,\,\,\,\,\,\,\,\,\,\,\,\,\,\,\,\,\,\,\,\,\,\,\,\,\,\,\,\,\,\,\,\,\,\,\,\,\,\,
\otimes 
\big(\big|_{x_1=\la_1+\partial} v_1\big)\otimes\ldots \overset{i,j}{\curlyvee}\ldots
\otimes\big(\big|_{x_{n+1}=\la_{n+1}+\partial} v_{n+1}\big)
\Big) \notag \\
&+
\sum_{i<j} 
\epsilon_{\Ga}(i,j) 
(-1)^{n+i+j-1} 
Y^{\pi_{ij}(\Ga)}_{\la_i+\la_j,\la_1,\ldots \overset{i,j}{\curlyvee}\ldots,\la_{n+1}} 
(v_iv_j \otimes v_1\otimes\ldots \overset{i,j}{\curlyvee}\ldots\otimes v_{n+1})\,,
\end{align}
where $X(i)$ is the sum of the variables $x_k$ with $k\neq i$ in the same connected component as the vertex $i$.

\begin{theorem}\label{thm:differential}
Formula \eqref{eq:explicitform} defines a differential on the space of classical cochains 
$C_{\cl}=\bigoplus_{n\geq0}C_{\cl}^{n}$, i.e. $d^2=0$.
\end{theorem}
\begin{proof}
As we will see in Section \ref{sec:3.4},
formula \eqref{eq:explicitform} corresponds to the differential of the classical PVA cohomology
defined in \cite{BDSHK18} with an operadic aproach.
\end{proof}

\begin{remark}\label{thm:PVAstructure}
The Poisson vertex algebra structure on $\mc V$ defines an element $X\in C_{\cl}^2$
by
\begin{equation}\label{eq:PVAstructure}
X^{\bullet\!-\!\!\!\!\to\!\bullet}(a\otimes b)
=
ab
\,\,,\,\,\,\,
X^{\bullet\,\,\bullet}_{\lambda,-\lambda-\partial}(a\otimes b)
=
[a_\lambda b]
\,.
\end{equation}
The skewsymmetry of $X$ is equivalent to the commutativity of $ab$ and the skewsymmetry of $[a_\lambda b]$,
while the sesquilinearity of $X$ is equivalent to the sesquilinearity of $[a_\lambda b]$
and the fact that $\partial$ is a derivation of $ab$.
Moreover, the associativity for $ab$, the Jacobi identity for $[a_\lambda b]$
and the Leibniz rule relating them,
together are equivalent to the condition that $dX=0$,
see \cite[Thm.10.7]{BDSHK18}.
\end{remark}
\begin{example}\label{ex:1}
Consider the completely disconnected graph $\Gamma=\bullet\,\,\bullet\,\cdots\,\bullet$.
Then in formula \eqref{eq:explicitform}, 
all $\deg_\Gamma(h)$, $\epsilon_\Gamma(i,j)$ and $X(i)$ vanish, and we obtain
\begin{align*}
&(dY)^{\bullet\,\cdots\,\bullet}_{\la_1,\ldots,\la_{n+1}}(v_1\otimes\ldots\otimes v_{n+1}) \notag \\
&=
\sum_{h=1}^{n+1} (-1)^{n-h} \Big[{v_h}_{\la_h} 
Y^{\bullet\,\cdots\,\bullet}_{\la_1,\ldots\overset{h}{\curlyvee}\ldots,\la_{n+1}} 
(v_1\otimes\ldots \overset{h}{\curlyvee}\ldots\otimes v_{n+1}) \Big] \notag \\
&+
\sum_{1\leq i<j\leq n+1} 
(-1)^{n+i+j-1} 
Y^{\bullet\,\cdots\,\bullet}_{\la_i+\la_j,\la_1,\ldots\overset{i,j}{\curlyvee}\ldots,\la_{n+1}} 
\Big([{v_i}_{\la_i}v_j]\otimes v_1\otimes\ldots \overset{i,j}{\curlyvee}\ldots \otimes v_{n+1}\Big)
\,,
\end{align*}
which is the same as \eqref{eq:lca-d}.
\end{example}
\begin{example}\label{ex:2}
Consider the case when $\Gamma=\Gamma_{n+1}$ is the standard $(n+1)$-line \eqref{eq:Ln}.
Then $\deg_\Gamma(h)=1$ for the endpoints $h=1$ or $n+1$, $\deg_\Gamma(h)=2$ otherwise,
so that the first sum in \eqref{eq:explicitform} vanishes.
The third sum vanishes as well
because, when $\epsilon_\Gamma(i,j)=0$, the graph $\pi_{ij}(\Gamma)$ has a cycle.
In the fourth sum we only have the terms with $j=i+1$.
Thus we obtain
\begin{align*}
&(dY)^{\Ga_{n+1}}(v_1\otimes\cdots\otimes v_{n+1}) \\
& =
(-1)^{n+1} 
v_1 
Y^{\Ga_n}
(v_2\otimes\cdots\otimes v_{n+1}) 
+
Y^{\Ga_n}
(v_1\otimes\cdots\otimes v_{n}) v_{n+1} \\
&+
\sum_{i=1}^n 
(-1)^{n+i-1} 
Y^{\Ga_n}
(v_1\otimes\cdots \otimes v_iv_{i+1}\otimes \cdots\otimes v_{n+1})\,.
\end{align*}
For the last term we used the skewsymmetry of $Y$ to bring the factor $v_iv_{i+1}$ in position $i$.
This is the formula for the Hochschild differential \cite{Hoc45}.
\end{example}

\subsection{Proof of the formula for the differential}\label{sec:3.4}

In the present paper, the formula \eqref{eq:explicitform}
for the classical PVA cohomology differential $d$ is taken as a definition.
Here, we show how that formula
is derived from the approach of \cite{BDSHK18}.
%
This implies Theorem \ref{thm:differential}.

Recall from \cite[Sec.10]{BDSHK18} the classical operad $\mc P_{\cl}(\Pi\mc V)$,
defined as follows.
The space $\mc P_{\cl}(\Pi\mc V)(n)$ consists of maps \eqref{eq:operadcl}
satisfying the cycle relations \eqref{eq:cycle1}
and the sesquilinearity conditions \eqref{eq:sesq1}-\eqref{eq:sesq2}.
There is a natural action of the symmetric group $S_n$ on $\mc P_{\cl}(\Pi\mc V)(n)$
defined by simultaneously permuting all the $\lambda_i$'s, the vectors $v_i$'s
and the vertices of the graph $\Gamma$,
and multiplying by the sign of the permutation,
since all vectors in $\Pi\mc V$ are odd.
Explicitly (see \cite[Eq.(10.10)]{BDSHK18})
\begin{equation}\label{eq:actSn}
(Y^\sigma)^{\Gamma}_{\lambda_1,\dots,\lambda_n}
(v_1\otimes\cdots\otimes v_n)
=
\sign(\sigma)
Y^{\sigma(\Gamma)}_{\lambda_{\sigma^{-1}(1)},\dots,\lambda_{\sigma^{-1}(n)}}
(v_{\sigma^{-1}(1)}\otimes\cdots\otimes v_{\sigma^{-1}(n)})
\,.
\end{equation}
Then the skewsymmetry conditions \eqref{eq:actclassicoperad}
are equivalent to the $S_n$ invariance of $Y$.
Therefore
\begin{equation}\label{20201022:eq1}
C^n_{\cl}=W^{n-1}_{\cl}(\Pi\mc V)=\big(\mc P_{\cl}(\Pi\mc V)(n)\big)^{S_n}
\end{equation}
is the space of fixed points under the action of the symmetric group $S_n$
in the classical operad $\mc P_{\cl}(\Pi\mc V)$.

The composition products in $\mc P_{\cl}(\Pi\mc V)$
are given by \cite[Eq.(10.11)]{BDSHK18}.
Here we need the special case of $\circ_1$-product (see \cite[Rem.10.3 and Eq.(8.18)]{BDSHK18}).
For $A\in \mc P_{\cl}(k)$, $B\in \mc P_\cl(m)$ and $G\in\mc G(m+k-1)$,
the $\circ_1$-product $A \circ_1 B \in \mc P_{\cl}(m+k-1)$ is given by 
\begin{equation}\label{com4}
\begin{split}
& (A \circ_1 B)_{\la_1,\dots,\la_{m+k-1}}^G (v_1\otimes\dots\otimes v_{m+k-1} ) \\
& = A_{\la_{G'},\la_{m+1},\dots,\la_{m+k-1}}^{\bar G''} \bigl( 
B_{\la_1+\la_{G_1}+\dd_{G_1},\dots,\la_{m}+\la_{G_{m}}+\dd_{G_{m}}}^{G'} (v_1\otimes\cdots \\
&\qquad\qquad\qquad\qquad\qquad
\dots\otimes v_{m})
\otimes v_{m+1}\otimes\dots\otimes v_{m+k-1} \bigr)
\,.
\end{split}
\end{equation}
Here $G'$ is the subgraph of $G$ with vertices $1,\dots,m$ and all edges from $G$ among these vertices;
$G''$ is the subgraph of $G$ that includes all edges of $G$ not in $G'$;
and $\bar G''$ is the graph with vertices labeled $1,m+1,\dots,m+k-1$ and edges obtained from the edges of $G''$ by replacing any vertex $1\le i\le m$ with $1$, keeping the same orientation.
Finally, the graph $G_i$ $(1\le i\le m)$ is the subgraph of $G''$ obtained from
the connected component of the vertex $i$ in $G''$ by removing from it the vertex $i$ 
and all edges connected to $i$.

By  \cite[Thm.3.4]{BDSHK18}, $W_{\cl}(\Pi\mc V)=\bigoplus_{k\geq-1}W^k_{\cl}(\Pi\mc V)$ has the structure of a 
$\mb Z$-graded Lie superalgebra.
In particular, for $X\in W^1_\cl(\Pi\mc V)$ and $Y\in W^{n-1}_\cl (\Pi\mc V)$,
their Lie bracket is given by \cite[Eqs. (3.13), (3.16)]{BDSHK18}:
\begin{equation} \label{eq:nliebracket}
[X,Y] = \sum_{\sigma \in S_{n,1}}(X\circ_1 Y)^{\sigma^{-1}} 
+ (-1)^n \sum_{\tau \in S_{2,n-1}}(Y\circ_1 X)^{\tau^{-1}}
\,,
\end{equation}	
where $S_{n,1}$ and $S_{2,n-1}$ denote the sets of shuffles from Section \ref{sec:shuffles}.

The element $X\in C^2_{\cl}=W^1_{\cl}(\Pi\mc V)$ in \eqref{eq:PVAstructure}
is odd and satisfies $[X,X]=0$, see \cite[Thm.10.7]{BDSHK18}.
Hence, $(\ad X)^2=0$,
and $d=\ad X$ was taken as the differential of the classical PVA cohomology complex 
in \cite[Def.10.8]{BDSHK18}.
As a consequence, the classical PVA cohomology $H_{\cl}(\mc V)$
has an induced Lie superalgebra structure.
Here we show that the differential $d$ in \eqref{eq:explicitform}
coincides with $\ad X$ from \eqref{eq:nliebracket}:

\begin{proposition}
For $Y\in C_{\cl}^n=W^{n-1}_\cl (\Pi\mc V)$, we have $dY=[X,Y]$.
\end{proposition}
\begin{proof}
Recalling from Section \ref{sec:shuffles} the definition of shuffles, we have $S_{n,1}=\{\sigma_h\}_{h=1}^{n+1}$
where
\begin{equation*}
\sigma_h = \left(\begin{array}{ccc|c}
1 & \cdots\cdots\cdots & n  & n+1 \\
1 &  \cdots\overset{h}{\curlyvee}\cdots & n+1 & h
\end{array} \right) \,,
\end{equation*}
and $S_{2,n-1}=\{\tau_{i,j}\}_{1\leq i<j\leq n+1}$ where
\begin{equation*} 
\tau_{i,j} = \left(\begin{array}{cc|ccc}
1 & 2  & 3 & \cdots\cdots\cdots  & n+1 \\
i & j & 1 & \cdots\overset{i,j}{\curlyvee}\cdots  & n+1
\end{array} \right)
\,.
\end{equation*}
Clearly,
$$
\sign(\sigma_h)=(-1)^{n-h+1}
\,\,\text{ and }\,\,
\sign(\tau_{i,j})=(-1)^{i+j-1}
\,.
$$
Hence, formula \eqref{eq:nliebracket} becomes
\begin{align} 
\label{eq:nlb}
& [X,Y]^{\Ga}_{\la_1,\dots,\la_{n+1}} (v_1\otimes\cdots\otimes v_{n+1}) 
\\
& =\sum_{h=1}^{n+1} ((X\circ_1 Y)^{\sigma_h^{-1}})^{\Ga}_{\la_1,\dots,\la_{n+1}}(v_1\otimes\cdots\otimes v_{n+1}) 
\notag\\
&\;\;\;\; 
+ (-1)^n \sum_{i<j} ((Y\circ_1 X)^{\tau_{i,j}^{-1}})^{\Ga}_{\la_1,\dots,\la_{n+1}}(v_1\otimes\cdots\otimes v_{n+1})
\notag\\
&=\sum_{h=1}^{n+1} (-1)^{n-h+1} (X\circ_1 Y)^{\sigma_h^{-1}(\Ga)}_{\la_1,\ldots \overset{h}{\curlyvee} \ldots,\la_{n+1},\la_h}(v_1\otimes\ldots\overset{h}{\curlyvee}\ldots\otimes v_{n+1}\otimes v_h) \notag \\
&\;\;\;\;
+ \sum_{i<j} (-1)^{n+i+j-1}(Y\circ_1 X)^{\tau_{i,j}^{-1}(\Ga)}_{\la_i,\la_j,\la_1,\ldots\overset{i,j}{\curlyvee}\ldots,\la_{n+1}}(v_i\otimes v_j\otimes v_1\otimes\ldots\overset{i,j}{\curlyvee}\ldots\otimes v_{n+1}) \,,
\notag
\end{align}
where $\overset{h}{\curlyvee}$ denotes a missing factor.

Let us study the two summands in the right-hand side of \eqref{eq:nlb} separately.
To compute the first summand,
we use equation \eqref{com4} with $A=X$, $B=Y$, $k=2$, $m=n$
and $G={\sigma_h^{-1}(\Ga)}$.
Note that  $\sigma_h^{-1}(\Ga)$ is obtained by moving the $h$-th vertex at the end of the graph.
Hence, 
$G'=\Gamma\backslash h$
and $G''$ is the subgraph of $\Gamma$ obtained by keeping only the edges in or out of the vertex $h$.
Then $\bar G''$ is a graph with two vertices labeled $1$ and $h$,
and 
$$
\bar G''
= 
\begin{cases}
{}\,\bullet\,\,\,\,\,\,\bullet & \text{ if  } \deg_{\Gamma}(h)=0\,, \\
{}_1\bullet\!\to\!\!\bullet_h & \text{ if  } \deg_{\Gamma}(h)=\deg^-_{\Gamma}(h)=1\,, \\
{}_1\bullet\!\leftarrow\!\!\bullet_h & \text{ if  } \deg_{\Gamma}(h)=\deg^+_{\Gamma}(h)=1\,,
\end{cases}
$$
and $\bar G''$ has a cycle if $\deg_{\Gamma}(h)\geq2$.
Moreover, if $\deg_{\Gamma}(h)=0$ then $G_i=\emptyset$ for all $i$,
while  if $\deg_{\Gamma}(h)=1$ and there is an edge connecting $h$ with $j$
then $G_i=\emptyset$ for all $i\neq j$ and $G_j=\bullet_h$.
As a result, provided that $\deg_{\Gamma}(h)\leq 1$, we obtain
\begin{align}
& (X\circ_1 Y)^{\sigma_h^{-1}(\Ga)}_{\la_1,\ldots \overset{h}{\curlyvee} \ldots,\la_{n+1},\la_h}(v_1\otimes\ldots\overset{h}{\curlyvee}\ldots\otimes v_{n+1}\otimes v_h) 
\label{eq:corona1} \\
& = 
\begin{cases}
X_{-\lambda_h-\partial,\la_h}^{\bullet\,\,\bullet} \bigl( 
Y^{\Gamma\backslash h}_{\la_1,\ldots\overset{h}{\curlyvee}\ldots,\la_{n+1}} 
(v_1\otimes\ldots \overset{h}{\curlyvee}\ldots\otimes v_{n+1})
\otimes v_h
\bigr)
& \text{ if  } \deg_{\Gamma}(h)=0\,, \\
X^{\bullet\to\bullet} \bigl( 
Y^{\Ga\backslash h}_{\la_1,\ldots\overset{h}{\curlyvee}\ldots,\la_j+\lambda_h+\partial_h,\ldots,\la_{n+1}} 
\!\!\!\!\!\!\!\!\!
(v_1\otimes\ldots\overset{h}{\curlyvee}\ldots\otimes v_{n+1}) 
\otimes v_h
\bigr)
& \text{ if  } j\rightarrow h\in E(\Ga)\,, \\
X^{\bullet\leftarrow\bullet} \bigl( 
Y^{\Ga\backslash h}_{\la_1,\ldots\overset{h}{\curlyvee}\ldots,\la_j+\lambda_h+\partial_h,\ldots,\la_{n+1}} 
\!\!\!\!\!\!\!\!\!
(v_1\otimes\ldots\overset{h}{\curlyvee}\ldots\otimes v_{n+1}) 
\otimes v_h
\bigr)
& \text{ if  } j\leftarrow h\in E(\Ga)\,,
\end{cases}
\notag
\end{align}
where $\partial_h$ denotes the action of $\partial$ on $v_h$,
while $(X\circ_1 Y)^{\sigma_h^{-1}(\Ga)}=0$ if $\deg_{\Gamma}(h)>1$.

To compute the second summand in the right-hand side of \eqref{eq:nlb},
we use equation \eqref{com4} with $A=Y$, $B=X$, $k=n$, $m=2$
and $G=\tau_{i,j}^{-1}(\Gamma)$.
Note that $\tau_{i,j}^{-1}(\Ga)$ is obtained by moving vertices $i$ and $j$ at the beginning of $\Ga$, 
keeping the order between $i$ and $j$.
Hence, 
$$
G'
= 
\begin{cases}
\bullet\,\,\,\,\,\,\bullet & \text{ if  there is no edge between $i$ and $j$ in $\Gamma$} \,, \\
\bullet\!\to\!\!\bullet & \text{ if  } i\to j\,\in E(\Gamma)\,, \\
\bullet\!\leftarrow\!\!\bullet & \text{ if  }  i\leftarrow j\,\in E(\Gamma)\,,
\end{cases}
$$
while $\bar G''=\pi_{ij}(\Gamma)$. 
As a result, we obtain
\begin{align}
& (Y\circ_1 X)^{\tau_{i,j}^{-1}(\Ga)}_{\la_i,\la_j,\la_1,\ldots\overset{i,j}{\curlyvee}\ldots,\la_{n+1}}
(v_i\otimes v_j\otimes v_1\otimes\ldots\overset{i,j}{\curlyvee}\ldots\otimes v_{n+1})
\label{eq:corona2} \\
& = 
\begin{cases}
Y_{\la_i+\la_j,\la_1,\dots\overset{i,j}{\curlyvee}\ldots \la_{n+1}}^{\pi_{ij}(\Gamma)} 
\!\!\!\!\!\!\!
\bigl( 
X_{\la_i+\la_{G_i}+\dd_{G_i},\la_{j}+\la_{G_{j}}+\dd_{G_{j}}}^{\bullet\,\,\bullet} \!\! (v_i\otimes v_{j})
\otimes v_1\otimes \dots\overset{i,j}{\curlyvee} \dots \otimes v_{n+1}
\bigr)  \\
\qquad\qquad\qquad\qquad\qquad\qquad\qquad\qquad\qquad\qquad\qquad\qquad
\text{ if  }\,\, \epsilon_{\Gamma}(i,j)=0\,, \\
Y_{\la_i+\la_j,\la_1,\dots\overset{i,j}{\curlyvee}\ldots \la_{n+1}}^{\pi_{ij}(\Gamma)} \bigl( 
X^{\bullet\to\bullet} (v_i\otimes v_{j})
\otimes v_1\otimes \dots\overset{i,j}{\curlyvee} \dots \otimes v_{n+1}
\bigr) \\
\qquad\qquad\qquad\qquad\qquad\qquad\qquad\qquad\qquad\qquad\qquad\qquad
\text{ if  }\,\, \epsilon_{\Gamma}(i,j)=1\,, \\
Y_{\la_i+\la_j,\la_1,\dots\overset{i,j}{\curlyvee}\ldots \la_{n+1}}^{\pi_{ij}(\Gamma)} \bigl( 
X^{\bullet\leftarrow\bullet} (v_i\otimes v_{j})
\otimes v_1\otimes \dots\overset{i,j}{\curlyvee} \dots \otimes v_{n+1}
\bigr) \\
\qquad\qquad\qquad\qquad\qquad\qquad\qquad\qquad\qquad\qquad\qquad\qquad
\text{ if  }\,\, \epsilon_{\Gamma}(i,j)=-1 \,.
\end{cases}
\notag
\end{align}
Combining equations \eqref{eq:nlb}, \eqref{eq:corona1} and \eqref{eq:corona2}
and recalling \eqref{eq:PVAstructure}, we obtain \eqref{eq:explicitform}.
\end{proof}

\section{The Main Theorem}\label{sec:4}

To a Poisson vertex algebra $\mc V$ we associate two cohomology complexes: 
the variational PVA cohomology complex $C_{\PV}$ introduced in Section \ref{sec:2.2}, 
and the classical PVA cohomology complex $C_{\cl}$ introduced in Section \ref{sec:3}.
Recall also from Remark \ref{rem:Wnotation} and Section \ref{sec:3.4},
that these complexes have the structure of a Lie superalgebra.
It is natural to ask what is the relation between these two cohomology theories.
A partial answer was provided by the following:
\begin{theorem}[{\cite[Theorem\ 11.4]{BDSHK18}}]\label{thm:homomorphism}
We have a canonical injective homomorphism of Lie superalgebras
\begin{equation}\label{eq:homomorphism}
H_{\PV}(\mc V)
\,\hookrightarrow\,
H_{\cl}(\mc V)
\end{equation}
induced by the map that sends $f\in C_{\PV}^n$ to $Y\in C_\cl^n$
such that 
$$
Y^{\bullet\,\cdots\,\bullet}= f
\,\,\text{ and }\,\,
Y^\Gamma=0
\,\text{ if }\,
|E(\Gamma)|\neq\emptyset
\,.
$$
\end{theorem}
It was left as an open question in \cite{BDSHK18} whether \eqref{eq:homomorphism} is, in fact, 
an isomorphism. The main result of this paper will be the proof that this is indeed the case,
under some regularity assumption on $V$. 
\begin{theorem} \label{thm:main}
Assuming that the PVA $\mc V$, as a differential algebra, is a finitely-generated algebra of differential polynomials,
the Lie superalgebra homomorphism \eqref{eq:homomorphism} is an isomorphism. 
\end{theorem}
The remainder of the paper will be devoted to the proof of Theorem \ref{thm:main}.
In Section \ref{sec:5}, we introduce a new cohomology complex, called the
sesquilinear Harrison cohomology complex.
In Section \ref{sec:6}, we define a filtration of the classical PVA cohomology complex
and we prove that its associated graded is isomorphic to the sesquilinear
Harrison cohomology complex.
We then show, in Section \ref{sec:7} that the cohomology of the sesquilinear
Harrison cohomology  complex
vanishes in positive degree.
Using that, we complete, in Section \ref{sec:8}, the proof of Theorem \ref{thm:main}.

\section{Sesquilinear Harrison cohomology}\label{sec:5}

In the present Section we introduce the sesequilinear Hochschild and Harrison cohomology complexes.
In order to do so, we first review the differential Hochschild 
and Harrison cohomology complexes.

\subsection{Differential Hochschild cohomology complex} \label{ssec:Hoch}

Let $A$ be an associative algebra over the base field $\mb F$,
and $M$ be an $A$-bimodule. 
The corresponding \emph{Hochschild cohomology complex} of $A$ with coefficients in $M$ 
is defined as follows \cite{Hoc45}.
The space of $n$-cochains is
\begin{equation} \label{eq:hochobject}
\Hom(A^{\otimes n},M)\;,
\end{equation}
and the differential $d\colon\Hom(A^{\otimes n},M) \rightarrow \Hom(A^{\otimes n+1},M)$ is defined by
\begin{align} \label{eq:hochdifferential}
(df)(a_1\otimes &\dots\otimes a_{n+1}) = a_1f(a_2\otimes\dots\otimes a_{n+1}) \notag \\ &+\sum_{i=1}^{n} (-1)^i f(a_1\otimes \dots\otimes a_{i-1}\otimes a_ia_{i+1}\otimes a_{i+2}\otimes \dots\otimes a_{n+1}) \notag \\
&+ (-1)^{n+1} f(a_1\otimes\dots\otimes a_n)a_{n+1}\,.
\end{align}
%

If $A$ is an associative algebra with a derivation $\partial\colon A\rightarrow A$, 
and $M$ is a differential bimodule over $A$
(i.e., the action of $\partial$ is compatible with the bimodule structure), 
we may consider the \emph{differential Hochschild cohomology complex} 
by taking the subspace of $n$-cochains
\begin{equation}\label{eq:cochaindiff}
\Hom_{\mb F[\partial]}(A^{\otimes n}, M)\,.
\end{equation}
It is clear by the definition \eqref{eq:hochdifferential} that the differential $d$ 
maps $\Hom_{\mb F[\partial]}(A^{\otimes n}, M)$ to $\Hom_{\mb F[\partial]}(A^{\otimes n+1}, M)$. 
Hence, we have a cohomology subcomplex.

\subsection{Differential Harrison cohomology complex} \label{ssec:Harrison}

Let us now recall Harrison's original definition of his cohomology complex \cite{Har62}, see also \cite{GS87,L13}. 
Let $A$ be a commutative associative algebra, and $M$ be a symmetric $A$-bimodule, 
i.e., such that $am=ma$, for all $a\in A$ and $m\in M$. 
For every $1<k\leq n$
define the following endomorphism on the space $\Hom(A^{\otimes n},M)$:
\begin{equation} \label{eq:lkf}
(L_kF)(a_1\otimes\dots\otimes a_n):= \sum_{\pi \in \mc M_n^k} (-1)^{dr(\pi)} F(a_{\pi(1)}\otimes\dots\otimes a_{\pi(n)})\,,
\end{equation}
where $\mc M_n^k$ is the set of monotone permutations starting at $k$,
defined in Section \ref{ssec:monotone}.

A \emph{Harrison $n$-cochain} is defined as a Hochschild $n$-cochain
$F\in\Hom(A^{\otimes n},M)$ fixed by all operators $L_k$:
\begin{equation} \label{eq:harrisoncond}
L_kF=F\,, \;\text{ for every } \; 2\leq k\leq n\,.
\end{equation}
We will denote by 
\begin{equation} \label{eq:harcochain}
C^n_{\Har}(A,M)\subset \Hom(A^{\otimes n},M)
\end{equation}
the space of Harrison $n$-cochains.

Furthermore, if $A$ is a differential algebra with a derivation $\partial:A\rightarrow A$, 
and $M$ is a symmetric differential bimodule, we may consider the space of 
\emph{differential Harrison} $n$-cochains
\begin{equation}\label{eq:harrisondiff}
C^{n}_{\partial,\Har}(A,M)\subset \Hom_{\mb{F}[\partial]}(A^{\otimes n},M)\,,
\end{equation}
again defined by Harrison's conditions \eqref{eq:harrisoncond}.
\begin{proposition}[{\cite{GS87,BDSKV19}}] \
\begin{enumerate}[(a)] 
\item
The Harrison cohomology complex $(C_{\Har}(A,M),d)$ is a subcomplex of the
Hochschild cohomology complex.
\item
If $A$ is a differential algebra, with a derivation $\partial:A\rightarrow A$, 
the differential Harrison cohomology complex $(C_{\partial,\Har}(A,M),d)$ is a subcomplex 
of the differential Hochschild cohomology complex.
\end{enumerate}
\end{proposition}
The cohomology of the complex $(C_{\partial,\Har}(A,M),d)$
is the \emph{differential Harrison cohomology} of $A$ with coefficients in $M$,
and is denoted by $H_{\partial,\Har}(A,M)$.
Clearly, $H^0_{\partial,\Har}(A,M)=M$
and $H^1_{\partial,\Har}(A,M)=\Der^\partial(A,M)$
is the space of all derivations from $A$ to $M$ commuting with $\partial$.
\begin{remark}\label{rem:har}
It follows from \cite{GS87} that $H^n_{\partial,\Har}(A,M)$
is a direct summand of the differential Hochschild cohomology, for $n\geq2$.
\end{remark}

\subsection{The sesquilinear Hochschild cohomology complex} \label{ssec:sesquilinear-hh}

Let $\mc V$ be an associative differential algebra with derivation $\partial$,
and let $M$ be a differential bimodule over $\mc V$.
Fix $s\geq 1$ and let, as in Section \ref{ssec:lines},
$$
\ul k=(k_1,\dots,k_s)\,\in\mb Z_{\geq0}^s
\,,\,\,\,\,
K_0=0
\,,\,\,\,\,
K_t=k_1+\dots+k_t\,, \quad t=1,\dots,s
\,,
$$
and
$$
n=K_s=k_1+\dots+k_s
\,.
$$
Given $v_1,\dots,v_n\in\mc V$,
we denote 
\begin{equation}\label{eq:vupper}
v_{\ul k}^t=v_{K_{t-1}+1}\otimes\dots\otimes v_{K_t}\in\mc V^{\otimes k_t}
\,,\,\, t=1,\dots,s\,,
\end{equation}
so that 
\begin{equation}\label{eq:vupper2}
v:=v_1\otimes\dots\otimes v_n=v_{\ul k}^1\otimes\dots\otimes v_{\ul k}^s
\,\in\mc V^{\otimes n}
\,.
\end{equation}
Note that we allow $k_t$ to be $0$,
and in this case $v_{\ul k}^t=1\in\mb F$.

The $s$-\emph{sesquilinear Hochschild cohomology complex} $(C_{\sesq,\Hoc}^s(\mc V,M),d)$ 
of $\mc V$ with coefficients in $M$,
is defined as follows.
First we introduce the space $C^{\ul k}_{\Hoc}$ of all linear maps  \begin{equation}\label{20200623:eq1}
F_{\Lambda_1,\dots,\Lambda_s}\,:\,\,
\mc V^{\otimes n}\to M[\Lambda_1,\dots,\Lambda_s]/\langle\partial+\Lambda_1+\dots+\Lambda_s\rangle
\,,\,\,
v\mapsto F_{\Lambda_1,\dots,\Lambda_s}(v)
\,,
\end{equation}
satisfying the sesquilinearity conditions ($t=1,\dots,s$),
\begin{equation}\label{20200623:eq2}
F_{\Lambda_1,\dots,\Lambda_s}(v_{\ul k}^1\otimes\cdots \partial v_{\ul k}^t\cdots\otimes v_{\ul k}^s)
=
-\Lambda_t
F_{\Lambda_1,\dots,\Lambda_s}(v)
\,.
\end{equation}
For every $t=1,\dots,s$, we define the $t$-th differential
$d^{(t)}:\,C^{\ul k}_{\Hoc}\to C^{\ul k+\ul e_t}_{\Hoc}$,
where $\ul e_t$ is the $s$-tuple with all $0$ except for $1$ in position $t$,
given by
\begin{equation}\label{20200623:eq4}
\begin{split}
& (d^{(t)}F)_{\Lambda_1,\dots,\Lambda_s}(v_1\otimes\dots\otimes v_{n+1}) \\
& =
(-1)^{K_{t-1}}
\bigl(\big|_{x=\partial}v_{K_{t-1}+1}\bigr)
F_{\Lambda_1,\dots,\Lambda_t+x,\dots,\Lambda_s}
(v_1\otimes\dots\overset{K_{t-1}+1}{\curlyvee}\dots\otimes v_{n+1})  \\
& +
\sum_{i=K_{t-1}+1}^{K_t}
(-1)^i
F_{\Lambda_1,\dots,\Lambda_s}
(v_1\otimes\dots\otimes v_iv_{i+1}\otimes\dots\otimes v_{n+1}) \\
& +
(-1)^{K_t+1}
F_{\Lambda_1,\dots,\Lambda_t+x,\dots,\Lambda_s}
(v_1\otimes\dots\overset{K_{t}+1}{\curlyvee}\dots\otimes v_{n+1}) 
\bigl(\big|_{x=\partial}v_{K_{t}+1}\bigr)
\,.
\end{split}
\end{equation}
In other words, up to the overall sign $(-1)^{K_{t-1}}$
and up to the shift by $\partial$ in the variable $\Lambda_t$,
this is the Hochschild cohomology differential of $F$,
viewed as a function of 
$v_{\ul k+\ul e_t}^t
=v_{K_{t-1}+1}\otimes\dots\otimes v_{K_t+1}$,
considering all other vectors $v_{\ul k+\ul e_t}^{t'}$ with $t'\neq t$ as fixed parameters.
In equation \eqref{20200623:eq4} and throughout the rest of the paper,
the substitution $|_{x=\partial}$ means that the polynomial in $x$ is expanded,
$x$ is replaced by $\partial$, and it is applied, in this case, 
to the vector $v_{K_{t-1}+1}$ in the first term of the right-hand side,
and to the vector $v_{K_{t}+1}$ in the last term.
\begin{remark}\label{rem:1016}
Note that $M[\Lambda_1]/\langle\partial+\Lambda_1\rangle\simeq M$.
Using this, we identify the $1$-sesquilinear Hochschild cohomology complex
with the differential Hochschild cohomology complex, defined in Section \ref{ssec:Hoch}.
\end{remark}
\begin{remark}\label{rem:1022}
Note that, for $s>1$, by the sesquilinearity condition \eqref{20200623:eq2}, 
we have $C^{\ul k}_{\Hoc}=0$ if one of the $k_i$'s is zero.
\end{remark}
\begin{theorem}\label{20200625:prop1}
For each $\ul k\in\mb Z_{\geq 0}^s$, 
equation \eqref{20200623:eq4} gives 
well defined maps 
$$
d^{(t)}\colon C^{\ul k}_{\Hoc}\to C^{\ul k+\ul e_t}_{\Hoc}
\,\,,\,\,\,\,
t=1,\dots,s
\,,
$$
which are anticommuting differentials:
$$
d^{(t)}d^{(t')}=-d^{(t')}d^{(t)}
\,\text{ for all }\, t,t'=1,\dots,s
\,.
$$
Hence, we get a $\mb Z_{\geq0}^s$-graded $s$-complex,
\begin{equation}\label{20200625:eq1}
\Big(\bigoplus_{\ul k\in\mb Z_{\geq0}^s} C^{\ul k}_{\Hoc},\, d^{(1)},\dots,d^{(s)}\Big)
\,.
\end{equation}
As a consequence, letting
\begin{equation}\label{20200625:eq2}
C^{s,n}_{\Hoc}=\bigoplus_{\ul k\,:\,K_s=n}C^{\ul k}_{\Hoc}
\,\,\text{ and }\,\,
d=\sum_{t=1}^sd^{(t)}
\colon C^{s,n}_{\Hoc}\to C^{s,n+1}_{\Hoc}
\,,
\end{equation}
we get a cohomology complex $\big(C_{\sesq,\Hoc}^s(\mc V,M)=\bigoplus_{n\geq0}C^{s,n}_{\Hoc},d\big)$.
\end{theorem}
\begin{proof}
%
In order to prove that $d^{(t)}$ is well defined,
we first check that, if 
$$
F_{\Lambda_1,\dots,\Lambda_s}(v)
=
(\partial+\Lambda_1+\dots+\Lambda_s)
G_{\Lambda_1,\dots,\Lambda_s}(v)
\,,
$$
for every $v\in\mc V^{\otimes n}$,
then the right-hand side of \eqref{20200623:eq4}
lies in $\langle\partial+\Lambda_1+\dots+\Lambda_s\rangle$.
Indeed, using the fact that $\partial$ is a derivation of the product in $\mc V$,
the first term of the right-hand side is equal to
$$
(-1)^{K_{t-1}}
(\partial+\Lambda_1+\dots+\Lambda_s)
\Big(
G_{\Lambda_1,\dots,\Lambda_t+\partial,\dots,\Lambda_s}
(v_1\otimes\dots\overset{K_{t-1}+1}{\curlyvee}\dots\otimes v_n)_\to v_{K_{t-1}+1}
\Big)\,.
$$
The second and third term are similar.


%
Next, we check that $d^{(t)}F$ satisfies the sesquilinearity conditions \eqref{20200623:eq1}
for every $t'\in\{1,\dots,s\}$ in place of $t$ and for $\ul k+\ul e_t$ in place of $\ul k$.
Let $v=v_{\ul k+\ul e_t}^1\otimes\dots\otimes v_{\ul k+\ul e_t}^1$
be the factorization of $v\in\mc V^{\otimes(n+1)}$ as in \eqref{eq:vupper2}.
If $\partial$ acts on the factor $v_{\ul k+\ul e_t}^{t'}$ with $t'\neq t$,
then in each term of the right-hand side of \eqref{20200623:eq4}
we get a factor of $-\Lambda_{t'}$,
by the sesquilinearity of $F$.
In the case when $t'=t$ we observe that 
$$
v_{\ul k+\ul e_t}^t
=
v_{K_{t-1}+1} \otimes w
\,,\text{ where }\,
w=v_{K_{t-1}+2} \otimes\dots\otimes v_{K_{t}+1}
\,.
$$
Then
$$
\partial v_{\ul k+\ul e_t}^t
=
\partial v_{K_{t-1}+1} \otimes w
+
v_{K_{t-1}+1} \otimes \partial w
\,.
$$
Hence, if we replace $v_{\ul k+\ul e_t}^t$ by $\partial v_{\ul k+\ul e_t}^t$ in 
$(d^{(t)}F)_{\Lambda_1,\dots,\Lambda_s}(v)$,
the first term in the right-hand side of \eqref{20200623:eq4}
becomes, up to the sign $(-1)^{K_{t-1}}$,
\begin{align*}
& F_{\Lambda_1,\dots,\Lambda_t+\partial,\dots,\Lambda_s}
(v_1\otimes\cdots \partial w\dots\otimes v_n)_\to v_{K_{t-1}+1} \\
& +
F_{\Lambda_1,\dots,\Lambda_t+\partial,\dots,\Lambda_s}
(v_1\otimes\cdots w\dots\otimes v_n)_\to \partial v_{K_{t-1}+1} \\
& =
F_{\Lambda_1,\dots,\Lambda_t+\partial,\dots,\Lambda_s}
(v_1\otimes\cdots w\dots\otimes v_n)_\to (-\Lambda_t-\partial) v_{K_{t-1}+1} \\
& +
F_{\Lambda_1,\dots,\Lambda_t+\partial,\dots,\Lambda_s}
(v_1\otimes\cdots w\dots\otimes v_n)_\to \partial v_{K_{t-1}+1} \\
& =
-\Lambda_t F_{\Lambda_1,\dots,\Lambda_t+\partial,\dots,\Lambda_s}
(v_1\otimes\cdots w\dots\otimes v_n)_\to  v_{K_{t-1}+1} 
\,.
\end{align*}
The other two terms in \eqref{20200623:eq4} are similar,
proving the sesquilinearity of $d^{(t)}F$.

Next, we prove that $d^{(t)}$ and $d^{(t')}$ anticommute for all $t,t'$.
For $t'\neq t$, $d^{(t)}$ and $d^{(t')}$ act on a different set of variables,
hence, due to the overall signs, they anticommute.
For $t'=t$ we need to show that $(d^{(t)})^2=0$,
which is similar to the proof that the square of the Hochschild differential is zero.
%
%
For simplicity of notation, let us check this for $t=1$.
Then $K_1=k_1$ will be denoted simply as $k$.
Applying formula \eqref{20200623:eq4} twice, we obtain:
\allowdisplaybreaks
\begin{align*}
& (d^{(1)}(d^{(1)}F))_{\Lambda_1,\dots}(v_1\otimes\dots\otimes v_{k+2}\otimes\cdots) \\
& =
\bigl(\big|_{x=\partial}v_1\bigr) (d^{(1)}F)_{\Lambda_1+x,\dots}
(v_2\otimes\dots\otimes v_{k+2}\otimes\cdots) \\
& +
\sum_{i=1}^{k+1}
(-1)^i
(d^{(1)}F)_{\Lambda_1,\dots}
(v_1\otimes\dots\otimes v_iv_{i+1}\otimes\dots\otimes v_{k+2}\otimes\cdots) \\
& +
(-1)^{k+2}
(d^{(1)}F)_{\Lambda_1+x,\dots}
(v_1\otimes\dots\otimes v_{k+1}\otimes\cdots)
\bigl(\big|_{x=\partial}v_{k+2}\bigr) \\
& =
\bigl(\big|_{x=\partial}v_1\bigr) 
\bigl(\big|_{y=\partial}v_2\bigr) 
F_{\Lambda_1+x+y,\dots}
(v_3\otimes\dots\otimes v_{k+2}\otimes\cdots) \\
& +
\sum_{j=2}^{k+1}
(-1)^{j-1}
\bigl(\big|_{x=\partial}v_1\bigr) 
F_{\Lambda_1+x,\dots}
(v_2\otimes\dots\otimes v_jv_{j+1}\otimes\dots\otimes v_{k+2}\otimes\cdots) \\
& + (-1)^{k+1}
\bigl(\big|_{x=\partial}v_1\bigr) F_{\Lambda_1+x+y,\dots}
(v_2\otimes\dots\otimes v_{k+1}\otimes\cdots)
\bigl(\big|_{y=\partial}v_{k+2}\bigr) \\
& -
\bigl(\big|_{x=\partial}(v_1v_2)\bigr) 
F_{\Lambda_1+x,\dots}
(v_3\otimes\dots\otimes v_{k+2}\otimes\cdots) \\
& + \sum_{i=2}^{k+1} (-1)^{i}
\bigl(\big|_{x=\partial}v_1\bigr) 
F_{\Lambda_1+x,\dots}
(v_2\otimes \dots \otimes v_iv_{i+1}\otimes\dots\otimes v_{k+2}\otimes\cdots) \\
& + \sum_{i=3}^{k+1}\sum_{j=1}^{i-2}
(-1)^{i+j}
F_{\Lambda_1,\dots}
(v_1\otimes \dots \otimes v_jv_{j+1}\otimes\dots  \otimes v_iv_{i+1}\otimes\dots \otimes v_{k+2}\otimes\cdots) \\
& - \sum_{i=2}^{k+1}
F_{\Lambda_1,\dots}
(v_1\otimes \dots \otimes v_{i-1}v_iv_{i+1}\otimes\dots \otimes v_{k+2}\otimes\cdots) \\
& + \sum_{i=1}^{k}
F_{\Lambda_1,\dots}
(v_1\otimes \dots \otimes v_iv_{i+1}v_{i+2} \otimes\dots \otimes v_{k+2}\otimes\cdots) \\
& + \sum_{i=1}^{k-1}\sum_{j=i+2}^{k+1}
(-1)^{i+j-1}
F_{\Lambda_1,\dots}
(v_1\otimes \dots \otimes v_iv_{i+1}\otimes\dots  \otimes v_jv_{j+1}\otimes\dots \otimes v_{k+2}\otimes\cdots) \\
& +
\sum_{i=1}^{k}
(-1)^{i+k+1}
F_{\Lambda_1+x,\dots}
(v_1\otimes\dots\otimes v_iv_{i+1}\otimes\dots\otimes v_{k+1}\otimes\cdots)
\bigl(\big|_{x=\partial}v_{k+2}\bigr) \\
& +
F_{\Lambda_1+x,\dots}
(v_1\otimes\dots\otimes v_{k}\otimes\cdots)
\bigl(\big|_{x=\partial}(v_{k+1}v_{k+2})\bigr) \\
& +
(-1)^k
\bigl(\big|_{x=\partial}v_1\bigr) 
F_{\Lambda_1+x+y,\dots}
(v_2\otimes\dots\otimes v_{k+1}\otimes\cdots)
\bigl(\big|_{y=\partial}v_{k+2}\bigr) \\
& +
\sum_{i=1}^k
(-1)^{k+i}
F_{\Lambda_1+x,\dots}
(v_1\otimes\dots\otimes v_iv_{i+1}\otimes\dots\otimes v_{k+1}\otimes\cdots)
\bigl(\big|_{x=\partial}v_{k+2}\bigr) \\
& -
F_{\Lambda_1+x+y,\dots}
(v_1\otimes\dots\otimes v_{k}\otimes\cdots)
\bigl(\big|_{x=\partial}v_{k+1}\bigr)
\bigl(\big|_{y=\partial}v_{k+2}\bigr)
\,.
\end{align*}
An inspection of the right-hand side shows that all terms
pairwise cancel with each other.
%
%
The remaining assertions of the theorem are an immediate consequence.
\end{proof}

The symmetric group $S_s$ acts naturally on each $C^{s,n}_{\Hoc}$ as follows.
A permutation $\sigma\in S_s$ maps $C^{\ul k}_{\Hoc}\to C^{\sigma(\ul k)}_{\Hoc}$,
where we recall that $\sigma(\ul k)=(k_{\sigma^{-1}(1)},\dots,k_{\sigma^{-1}(s)})$.
Given $F\in C^{\ul k}_{\Hoc}$,
its image $F^\sigma\in C^{\sigma(\ul k)}_{\Hoc}$ is given by
\begin{equation}\label{20200625:eq3}
(F^\sigma)_{\Lambda_1,\dots,\Lambda_s}(v)
=
\pm
F_{\Lambda_{\sigma^{-1}(1)},\dots,\Lambda_{\sigma^{-1}(s)}}
(v_{\ul k}^{\sigma^{-1}(1)}\otimes\dots\otimes v_{\ul k}^{\sigma^{-1}(s)})
\,,
\end{equation}
where the sign in the right-hand side is
\begin{equation}\label{20200716:eq2}
\pm=(-1)^{\sum_{t<t'\,:\,\sigma(t')<\sigma(t)}k_tk_{t'}}
\,,
\end{equation}
which is the Koszul sign obtained by permuting  
vectors $v_1,\dots,v_n$, viewed as having odd parity, 
according to \eqref{20200625:eq3}.
Moreover, for every $\sigma\in S_s$ and $t=1,\dots,s$, we have
\begin{equation}\label{20200625:eq4}
\sigma\circ d^{(t)}
=
d^{(\sigma(t))}\circ\sigma
\,.
\end{equation}

\subsection{The sesquilinear Harrison cohomology complex}\label{ssec:sesquilinear-harrison}

Let $\mc V$ be a commutative associative differential algebra,
and $M$ be a differential symmetric $\mc V$-bimodule.
We define the $s$-\emph{sesquilinear Harrison cohomology complex} $(C_{\sesq,\Har}^s(\mc V,M),d)$ 
as a subcomplex of the sesquilinear Hochschild cohomology complex of $\mc V$ with coefficients in $M$.
First, let $C^{\ul k}_{\Har}$ be the subspace of $C^{\ul k}_{\Hoc}$ 
consisting of all linear maps $F_{\Lambda_1,\dots,\Lambda_s}$ as in equation \eqref{20200623:eq1}
satisfying, in addition to the sesquilinearity conditions \eqref{20200623:eq2},
the following Harrison conditions ($1\leq t\leq s,\,2\leq m\leq k_t$):
\begin{equation}\label{20200623:eq3}
L_m^{(t)}F
:=
\sum_{\pi \in \mc M_{k_t}^m} (-1)^{dr(\pi)} F_{\Lambda_1,\dots,\Lambda_s}
(v_{\ul k}^1\otimes\cdots \pi^{-1}(v_{\ul k}^t) \cdots\otimes v_{\ul k}^s)
=
F_{\Lambda_1,\dots,\Lambda_s}(v)
\,,
\end{equation}
where $\mc M_{k_t}^m$ is the set of monotone permutations in $S_{k_t}$ starting at $m$, cf. \eqref{eq:lkf}.
\begin{proposition}\label{20200922:prop}
For every $\ul k\in\mb Z_{\geq0}$, $1\leq t,t'\leq s$ and $2\leq m\leq k_t$
we have
$$
d^{(t)}L_m^{(t')}=L_m^{(t')}d^{(t)}
\,.
$$
In particular, we obtain a cohomology subcomplex
$C_{\sesq,\Har}^s(\mc V,M)$
of $C_{\sesq,\Hoc}^s(\mc V,M)$
given by 
\begin{equation}\label{20200625:eq2b}
C^{s,n}_{\Har}=\bigoplus_{\ul k\,:\,K_s=n}C^{\ul k}_{\Har}
\,\,\text{ and }\,\,
d=\sum_{t=1}^sd^{(t)}
\colon C^{s,n}_{\Har}\to C^{s,n+1}_{\Har}
\,.
\end{equation}
\end{proposition}
\begin{proof}
For $t\neq t'$,
the operators $d^{(t)}$ and $L_m^{(t')}$ commute
because they act on different sets of variables, $v_{\ul k}^t$ and $v_{\ul k}^{t'}$ respectively.
For $t=t'$, the equation $d^{(t)}L_m^{(t)}=L_m^{(t)}d^{(t)}$
holds by a straightforward computation,
which is similar to the proof that the Harrison cohomology complex is a subcomplex of the Hochschild complex,
see \cite{GS87}.
\end{proof}

\begin{proposition}\label{20200625:prop2}
Equation \eqref{20200625:eq3} gives a well defined action of the symmetric group $S_s$
on $C^{s,n}_{\Har}$, which maps $C^{\ul k}_{\Har}$ to $C^{\sigma(\ul k)}_{\Har}$.
Moreover,
$\sigma$ commutes with the differential $d$ in \eqref{20200625:eq2}.
\end{proposition}
\begin{proof}
Recall from the end of the previous subsection,
that we have an action $\sigma$ which maps $C^{\ul k}_{\Hoc}$
to $C^{\sigma(\ul k)}_{\Hoc}$.
We only need to check that this action preserves the Harrison conditions \eqref{20200623:eq3}.
This is true because $L_m^{(t)}$ acts on the vectors from the $t$-th group $v_{\ul k}^{(t)}$,
while $\sigma$ permutes the groups.
The claim that $\sigma$ commutes with $d$ follows from \eqref{20200625:eq4}.
\end{proof}
Thanks to Proposition \ref{20200625:prop2},
we get a cohomology subcomplex given by the $S_s$-invariants:
\begin{equation}\label{20200625:eq5}
\Big(
C_{\sym,\Har}^s(\mc V,M)
=
\bigoplus_{n\geq0}
(C^{s,n}_{\Har})^{S_s}
\,,\,
d
\Big)
\,,\qquad
s\geq1
\,.
\end{equation}
We will call this complex the \emph{symmetric} $s$-\emph{sesquilinear Harrison cohomology complex}
of $\mc V$ with coefficients in $M$.
The degenerate case $s=0$ corresponds to setting
$$
\ul k=\emptyset\,,\,\,
n=K_0=0\,,\,\,
v=1\in V^{\otimes0}=\mb F\,.
$$
In this case, the symmetric $(s=0)$-sesquilinear Harrison cohomology complex 
$C^{s=0}_{\sym,\Har}(\mc V,M)$ is concentrated in degree $n=0$
and it is equal to $M/\partial M$, with the zero differential.
\begin{remark}\label{rem:1016b}
As in Remark \ref{rem:1016},
we have $M[\Lambda_1]/\langle\partial+\Lambda_1\rangle\simeq M$,
and, using this, we identify the $(s=1)$-sesquilinear Harrison cohomology complex
with the differential Harrison cohomology complex, defined in Section \ref{ssec:Hoch}.
\end{remark}

\section{Relation between symmetric sesquilinear Harrison and classical PVA
cohomology complexes}\label{sec:6}

We introduce a filtration of the classical PVA complex $(C_{\cl},d)$ defined in Section \ref{sec:3.1}.
For a graph $\Gamma$ we let $s(\Gamma)$ be the number of connected components of $\Gamma$.
Recall that for an acyclic graph $\Gamma\in\mc G_0(n)$ with $n$ vertices, we have 
$s(\Gamma)=n-|E(\Gamma)|$.
We then let, for $s\in\mb Z$, 
\begin{equation}\label{20200702:eq1}
F_sC_{\cl}
=
\big\{Y\in C_{\cl}
\,\big|\,
Y^{\Gamma}=0 \text{ for every graph } \Gamma \text{ such that } s(\Gamma)<s
\big\}
\,.
\end{equation}
This defines a decreasing filtration of vector spaces.
Note that $F_sC_{\cl}=C_{\cl}$ for $s\leq0$,
because any graph $\Gamma$ with $n$ vertices and $|E(\Gamma)|>n$ has a cycle
and therefore $Y^\Gamma=0$ by definition.
The same argument also gives $F_1C^n_{\cl}=C^n_{\cl}$ for $n\geq1$,
because any non-empty  graph $\Gamma$ with $n$ vertices and $|E(\Gamma)|>n-1$ has a cycle.
However, $F_1C^0_{\cl}=0$, and
moreover, $F_sC_{\cl}^n=0$ for $s>n$, since $s(\Gamma)=n-|E(\Gamma)|\leq n$.
\begin{proposition}\label{20200702:prop}
The filtration \eqref{20200702:eq1} is preserved by the action of the differential $d$ defined by \eqref{eq:explicitform}.
\end{proposition}
\begin{proof}
For $Y\in F_sC_{\cl}^n$,
we need to prove that $dY\in F_sC_{\cl}^{n+1}$.
This means that, for any $\Gamma\in\mc G(n+1)$ such that $s(\Gamma)<s$,
we have $(dY)^\Gamma=0$.
Let us consider separately the four terms in the right-hand side of \eqref{eq:explicitform}.

First, if $\deg_\Gamma(h)=0$,
then $h$ is an isolated vertex of $\Gamma$
and $s(\Gamma\backslash h)=s(\Gamma)-1<s-1$,
so that $Y^{\Gamma\backslash h}=0$.
Second, if $\deg_\Gamma(h)=1$,
then $h$ is a leaf of $\Gamma$
and $s(\Gamma\backslash h)=s(\Gamma)<s$,
so that again $Y^{\Gamma\backslash h}=0$.
Third,  if $\epsilon_\Gamma(i,j)=0$,
then there is no edge connecting $i$ and $j$.
Hence when we collapse them into a single vertex,
either we get a loop in $\pi_{ij}(\Gamma)$, 
if $i$ and $j$ are in the same connected component of $\Gamma$,
or else $s(\pi_{ij}(\Gamma))=s(\Gamma)-1<s$.
In both cases $Y^{\pi_{ij}(\Gamma)}=0$.
Finally,  if $\epsilon_\Gamma(i,j)\neq0$,
then there is an edge connecting $i$ and $j$.
In this case $s(\pi_{ij}(\Gamma))=s(\Gamma)<s$,
and again $Y^{\pi_{ij}(\Gamma)}=0$.
In conclusion, all four terms in the right-hand side of \eqref{eq:explicitform}
vanish if $s(\Gamma)<s$, as claimed.
\end{proof}
As a consequence,
the $s$-degree component of the associated graded of the classical PVA complex
$$
\gr_sC_{\cl}
=
F_sC_{\cl}/F_{s+1}C_{\cl}
$$
is again a complex for any fixed $s\geq0$
with the induced action of the differential $d$.
Note that in the special case $s=0$ we have
$\gr_0C_{\cl}=C^0_{\cl}=V/\partial\mc V$,
which is concentrated in degree $n=0$.

By the proof of Proposition \ref{20200702:prop},
if $\Gamma\in\mc G(n+1)$ and $\deg_\Gamma(h)=0$,
then $s(\Gamma\backslash h)=s(\Gamma)-1$.
Moreover, if $\epsilon_\Gamma(i,j)=0$,
then either $\pi_{ij}(\Gamma)$ has a loop 
or $s(\pi_{ij}(\Gamma))=s(\Gamma)-1$.
As a consequence, for $Y\in F_sC_{\cl}^n$, the first and third term in the right-hand side 
of \eqref{eq:explicitform} vanish.
Therefore, we get the following explicit formula for the differential of 
$[Y]=Y+F_{s+1}C_{\cl}^n\in\gr^sC_{\cl}^n$,
evaluated at $\Gamma\in\mc G(n+1)$ with $s(\Gamma)=s$:
\begin{align} \label{eq:explicitform2}
&(d[Y])^{\Ga}_{\la_1,\ldots,\la_{n+1}}(v_1\otimes\ldots\otimes v_{n+1}) \notag \\
&=
\!\!\!
\sum_{\substack{h:\, \deg_{\Ga}(h)=1 \\ j:\, \epsilon_{\Ga}(j,h)\neq 0}} 
\!\!\!\!\!\!
(-1)^{\deg_{\Ga}^+(h)+n-h+1} 
Y^{\Ga\backslash h}_{\la_1,\ldots\overset{h}{\curlyvee}\ldots,\la_j+x,\ldots,\la_{n+1}} 
\!\!\!\!\!\!\!\!\!
(v_1\otimes\ldots\overset{h}{\curlyvee}\ldots\otimes v_{n+1}) 
\big(\big|_{x=\la_h+\partial} v_h\big) \notag \\
&+
\sum_{i<j} 
\epsilon_{\Ga}(i,j) 
(-1)^{n+i+j-1} 
Y^{\pi_{ij}(\Ga)}_{\la_i+\la_j,\la_1,\ldots \overset{i,j}{\curlyvee}\ldots,\la_{n+1}} 
(v_iv_j \otimes v_1\otimes\ldots \overset{i,j}{\curlyvee}\ldots\otimes v_{n+1})
\,.
\end{align}
\begin{theorem}\label{thm:whatever}
For every $s\geq0$ we have  an isomorphism of complexes
between the $s$-degree component of the associated graded of the classical PVA cohomology complex
and the symmetric $s$-sesquilinear Harrison cohomology complex:
\begin{equation}\label{eq:201016}
\gr_sC_{\cl}
\simeq
C_{\sym,\Har}^s
\,.
\end{equation}
Explicitly, for $Y\in F_sC_{\cl}^n$ and $\ul k\in\mb Z^s_{\geq0}$
such that $K_s=n$,
the image of the linear map 
$$
Y^{\Gamma_{\ul k}}\colon 
\mc V^{\otimes n}
\to
\mc V[\lambda_1,\dots,\lambda_n]
/\langle \partial+\lambda_1+\cdots+\lambda_n\rangle
$$
depends only on the sums 
\begin{equation}\label{20200707:eq2}
\Lambda_t
=
\lambda_{K_{t-1}+1}+\dots+\lambda_{K_t}
\,,\qquad
t=1,\dots,s
\,,
\end{equation}
and therefore it can be viewed as a linear map
$$
Y^{\Gamma_{\ul k}}\colon 
\mc V^{\otimes n}
\to
\mc V[\Lambda_1,\dots,\Lambda_s]
/\langle \partial+\Lambda_1+\cdots+\Lambda_s\rangle
\,.
$$
Then the isomorphism \eqref{eq:201016} maps
\begin{equation}\label{20200707:eq1}
Y+F_{s+1}C_{\cl}^n 
\mapsto
\sum_{\underline k\in\mb Z_{\geq0}^s\,:\,K_s=n} Y^{\Gamma_{\ul k}}
\in (C^{s,n}_{\Har})^{S_s}
\,.
\end{equation}
\end{theorem}
\begin{proof}
The case $s=0$ is obvious, so we shall assume $s\geq1$.
%
%
Clearly, for $\ul k\in\mb Z_{\geq0}^s$, we have $s(\Gamma_{\ul k})\leq s$.
Hence, for $Y\in F_{s+1}C_{\cl}^n$, we have $Y^{\Gamma_{\ul k}}=0$,
and the map \eqref{20200707:eq1} is well defined.

%
Next, we show that $Y^{\Gamma_{\ul k}}\in C^{\ul k}_{\Har}$ for each $\ul k\in\mb Z_{\geq0}^s$.
By the first sesquilinearity condition \eqref{eq:sesq1},
$Y^{\Gamma_{\ul k}}$ is a map $\mc V^{\otimes n}\to 
\mc V[\Lambda_1,\dots,\Lambda_s]/\langle\partial+\Lambda_1+\dots+\Lambda_s\rangle$,
since $\Lambda_t=\lambda_{(\Gamma_{\ul k})_t}$.
The second sesquilinearity condition \eqref{eq:sesq2} for $Y$
implies the sesquilinearity \eqref{20200623:eq2} of $Y^{\Gamma_{\ul k}}$.
Moreover, the Harrison conditions \eqref{20200623:eq3} for $Y^{\Gamma_{\ul k}}$
follow from Lemma \ref{lem:identity},
or more precisely from equation \eqref{eq:identity} applied to the $t$-th connected component 
$(\Gamma_{\ul k})_t=\Gamma_{k_t}$ of $\Gamma_{\ul k}$,
and the cycle relations \eqref{eq:cycle1} for $Y$.
Hence, $Y^{\Gamma_{\ul k}}\in C^{\ul k}_{\Har}$, as stated.

%
In order to check that the right-hand side of equation \eqref{20200707:eq1}
is invariant under the symmetric group $S_s$,
pick a permutation $\sigma\in S_s$ and consider its action on $Y^{\Gamma_{\ul k}}$,
for a fixed $\ul k\in\mb Z_{\geq0}^s$.
Using equation \eqref{20200625:eq3}, we find 
\begin{equation}\label{eq:bolsena5}
((Y^{\Gamma_{\ul k}})^\sigma)_{\Lambda_1,\dots,\Lambda_s}(v)
=
\pm
Y^{\Gamma_{\ul k}}_{\Lambda_{\sigma^{-1}(1)},\dots,\Lambda_{\sigma^{-1}(s)}}
(v_{\ul k}^{\sigma^{-1}(1)}\otimes\dots\otimes v_{\ul k}^{\sigma^{-1}(s)}) \\
\,.
\end{equation}
where $\pm$ is as in \eqref{20200716:eq2}.
Let $\widetilde{\sigma}\in S_n$ be the permutation 
\begin{equation}\label{eq:bolsena0}
\widetilde{\sigma}(K_{t-1}+i)
=
k_{\sigma^{-1}(1)}+\dots+k_{\sigma^{-1}(\sigma(t)-1)}+i
\,,\,\,
t=1,\dots,s,\,i=1,\dots,k_t
\,.
\end{equation}
This permutation is defined so that
\begin{equation}\label{eq:bolsena1}
v_{\widetilde{\sigma}^{-1}(1)}\otimes\dots\otimes v_{\widetilde{\sigma}^{-1}(n)}
=
v_{\ul k}^{\sigma^{-1}(1)}\otimes\dots\otimes v_{\ul k}^{\sigma^{-1}(s)}
\,.
\end{equation}
Indeed, we have, by \eqref{eq:vupper},
\begin{align*}
v_{\ul k}^{\sigma^{-1}(1)}&\otimes\dots\otimes v_{\ul k}^{\sigma^{-1}(s)}
=
(v_{K_{{\sigma^{-1}(1)}-1}+1}\otimes\dots\otimes v_{K_{\sigma^{-1}(1)}})
\otimes
(v_{K_{{\sigma^{-1}(2)}-1}+1}\otimes\dots
\\
&\dots\otimes v_{K_{\sigma^{-1}(2)}})
\otimes
\dots
\otimes
(v_{K_{{\sigma^{-1}(s)}-1}+1}\otimes\dots\otimes v_{K_{\sigma^{-1}(s)}})
\,.
\end{align*}
On the other hand, we obviously have
\begin{align*}
v_{\widetilde{\sigma}^{-1}(1)}&\otimes\dots\otimes v_{\widetilde{\sigma}^{-1}(n)}
=
(v_{\widetilde{\sigma}^{-1}(1)}\otimes\dots\otimes v_{\widetilde{\sigma}^{-1}(k_{\sigma^{-1}(1)})})
\otimes
(v_{\widetilde{\sigma}^{-1}(k_{\sigma^{-1}(1)}+1)}\otimes\dots
\\
&\dots\otimes 
v_{\widetilde{\sigma}^{-1}(k_{\sigma^{-1}(1)}+k_{\sigma^{-1}(2)})})
\otimes\dots\otimes
(v_{\widetilde{\sigma}^{-1}(k_{\sigma^{-1}(1)}+\dots+k_{\sigma^{-1}(s-1)}+1)}\otimes\dots
\\
&\dots\otimes 
v_{\widetilde{\sigma}^{-1}(k_{\sigma^{-1}(1)}+\dots+k_{\sigma^{-1}(s)})})
\,.
\end{align*}
The above two formulas match thanks to the definition \eqref{eq:bolsena0} of $\widetilde{\sigma}$
with $t$ replaced by $\sigma^{-1}(t)$.
Notice also that, for the same reason, (cf. \eqref{20200707:eq2})
\begin{equation}\label{eq:bolsena2}
\lambda_{\widetilde{\sigma}^{-1}(k_{\sigma^{-1}(1)}+\dots+k_{\sigma^{-1}(t-1)}+1)}
+\dots+
\lambda_{\widetilde{\sigma}^{-1}(k_{\sigma^{-1}(1)}+\dots+k_{\sigma^{-1}(t)})}
=
\Lambda_{\sigma^{-1}(t)}
\,,
\end{equation}
and
\begin{equation}\label{eq:bolsena3}
\widetilde{\sigma}(\Gamma_{\ul k})
=
\Gamma_{\sigma(\ul k)}
\,,
\end{equation}
or, equivalently, 
$\widetilde{\sigma}^{-1}(\Gamma_{\ul k})=\Gamma_{\sigma^{-1}(\ul k)}$.
We then use the skewsymmetry of $Y$ \eqref{eq:actclassicoperad} 
with respect to $\widetilde{\sigma}^{-1}$
evaluated on the graph $\Gamma_{\ul k}$:
\begin{equation}\label{eq:bolsena4}
Y^{\widetilde{\sigma}^{-1}(\Gamma_{\ul k})}_{\lambda_1,\dots,\lambda_n}
(v_1\otimes\cdots\otimes v_n)
=
\sign(\widetilde{\sigma})
Y^{\Gamma_{\ul k}}_{\lambda_{\widetilde{\sigma}^{-1}(1)},\dots,\lambda_{\widetilde{\sigma}^{-1}(n)}}
(v_{\widetilde{\sigma}^{-1}(1)}\otimes\cdots\otimes v_{\widetilde{\sigma}^{-1}(n)})
\,.
\end{equation}
Notice that the $\pm$ sign in \eqref{eq:bolsena5}
is precisely $\sign(\widetilde{\sigma})$.
Hence, combining equations \eqref{eq:bolsena5}--\eqref{eq:bolsena4},
we get
\begin{equation}
(Y^{\Gamma_{\ul k}})^\sigma
=
Y^{\Gamma_{\sigma^{-1}(\ul k)}}
\,.
\end{equation}
As a consequence,
the sum in the right-hand side of \eqref{20200707:eq1} is $S_s$-invariant,
as claimed.


Next, we observe that the map \eqref{20200707:eq1} is injective.
Indeed, if $\sum_{\underline k\in\mb Z_{\geq0}^s\,:\,K_s=n} Y^{\Gamma_{\ul k}}=0$
in $C^{s,n}_{\Har}=\bigoplus_{\ul k\,:\,K_s=n}C^{\ul k}_{\Har}$,
then $Y^{\Gamma_{\ul k}}=0$ for every $\ul k\in\mb Z_{\geq0}^s$,
and therefore, by Theorem \ref{thm:lines},
$Y^{\Gamma}=0$ whenever $s(\Gamma)\leq s$.
Hence, $Y\in F_{s+1}C^n_{\cl}$,
so its image in $\gr_sC^n_{\cl}$ is zero.


Now we prove that \eqref{20200707:eq1} is surjective.
Take an element 
$$
F=\sum_{\ul k}F^{\ul k}\in C^{s,n}_{\Har}=\bigoplus_{\ul k\,:\,K_s=n}C^{\ul k}_{\Har}
\,,
$$
which is invariant under the action of the symmetric group $S_s$.
We want to construct $Y\in F_s C^n_{\cl}$
such that $Y^{\Gamma_{\ul k}}=F^{\ul k}$ for every $\ul k\in\mb Z_{\geq0}^s$.
Note that, by Remark \ref{rem:1022},
we can restrict to $\ul k\in\mb Z_{>0}^s$.
In the degenerate case $s=1$ and $k_1=0$,
we have $n=0$ and in this case the claim is obvious.
First, we define $Y\in\mc P_{\cl}(\Pi\mc V)(n)$, see Section \ref{sec:3.4}.
Recall by Theorem \ref{thm:lines}
that the proper $\ul k$-lines $\Gamma\in\mc L(n)$, defined by \eqref{eq:unionlines}, \eqref{eq:min},
form a basis for the vector space $\mb F\mc G(n)/\mc R(n)$,
if $k_1\leq\dots\leq k_s$ and $i^l_1<i^{l+1}_1$ whenever $k_l=k_{l+1}$.
Hence, it is enough to define $Y^\Gamma$ for each proper $\ul k$-line $\Gamma$
satisfying these conditions.
Given such $\Gamma$, there is a permutation $\tau\in S_n$
such that $\Gamma=\tau(\Gamma_{\ul k})$,
and we set
\begin{equation}\label{20201022:eq2}
Y^\Gamma_{\lambda_1,\dots,\lambda_n}(v_1\otimes\dots\otimes v_n)
=
\sign(\tau)F^{\ul k}_{\Lambda_1,\dots,\Lambda_s}(v_{\tau(1)},\dots,v_{\tau(n)})
\,,
\end{equation}
where the $\Lambda_t$'s are as in \eqref{20200707:eq2}.
This is well defined, since if $\tau\in S_n$ fixes $\Gamma_{\ul k}$,
then $\tau=\widetilde{\sigma}$ for some $\sigma\in S_s$ fixing $\ul k$,
and in this case the right-hand side of \eqref{20201022:eq2}
equals $F^{\ul k}_{\Lambda_1,\dots,\Lambda_s}(v_1,\dots,v_n)$
by the $S_s$-symmetry of $F$.
The cycle relations \eqref{eq:cycle1} and the first sesquilinearity condition \eqref{eq:sesq1} 
on $Y$ hold by construction.
The second sesquilinearity condition \eqref{eq:sesq2} follows immediately from the 
sesquiinearity \eqref{20200623:eq2} of $F$.

We are left to check that the map $Y$ defined by \eqref{20201022:eq2}
satisfies the skewsymmetry \eqref{eq:actclassicoperad},
or equivalently, the $S_n$-invariance $Y=Y^\sigma$, $\sigma\in S_n$, with respect to the action
\eqref{eq:actSn}.
It is enough to check this separately in the cases when the permutation
only acts on the vertices of a single line,
or when it permutes the lines.
In the first case, the invariance condition reduces
to the case $s=1$, for which the sesquiinear Harrison complex is equivalent to the differential Harrison complex,
and the claim was proved in \cite[Lemm4.9]{BDSKV19}.
In the second case, when the permutation $\sigma$ permutes the lines,
the $\sigma$-invariance of $Y$ holds by construction.


Finally, we show that the map \eqref{20200707:eq1} commutes with the action 
of the differentials \eqref{eq:explicitform} and \eqref{20200625:eq2}.
Recall that the differential \eqref{eq:explicitform} induces, in the associated graded complex $\gr_s C_{\cl}$
the differential \eqref{eq:explicitform2}.
Let us evaluate the right-hand side of \eqref{eq:explicitform2} for $\Gamma=\Gamma_{\ul k}$.
In the first sum, $h$ is a vertex of degree $1$,
hence it must be the beginning or end point of one of the $s$ lines in $\Gamma_{\ul k}$,
and $j$ is the vertex adjacent to it in the line.
Hence, when $h$ is the first vertex of the $t$-th line, we get the first term of \eqref{20200623:eq4},
while when $h$ is the last vertex of the $t$-th line, we get the third term of \eqref{20200623:eq4}.
Furthermore, in the second sum of the right-hand side of \eqref{eq:explicitform2},
the only non-zero terms have $\epsilon_{\Gamma_{\ul k}}(i,j)=1$,
which means that $i$ and $j$ are consecutive vertices of the same line in $\Gamma_{\ul k}$.
When they are in the $t$-th line we recover the second term of \eqref{20200623:eq4}.
This completes the proof.
\end{proof}

\section{Vanishing of the sesquilinear Harrison cohomology}\label{sec:7}

In this section, we prove a vanishing theorem for the (symmetric)
sesquilinear Harrison cohomology, introduced in Section  \ref{ssec:sesquilinear-harrison}.
First, we recall some basic facts about the Hochschild homology and
cohomology, and a weak form of the Hochschild--Kostant--Rosenberg (HKR)
Theorem. Next, we give the proof by P. Etingof of an analogous statement
for the differential Hochschild cohomology.
We generalize this to the sesquilinear Hochschild cohomology, introduced
in Section \ref{ssec:sesquilinear-hh}, to derive the vanishing theorem
for the (symmetric) sesquilinear Harrison cohomology.

\subsection{The Bar complex}
Let $A$ be an associative $\mb{F}$-algebra. Its \emph{Bar-resolution}
$B_\bullet(A)$ is a complex of
$A\text{-}A$-bimodules with 
\begin{equation}
B_k(A) = \underbrace{A \otimes \cdots \otimes A}_{k+2 \text{--times}}, \qquad k
\geq 0,
\end{equation}
where the differential $d\colon B_k(A) \rightarrow B_{k-1}(A)$ is given by 
\[ d \bigl( a_0 \otimes \cdots \otimes a_{k+1} \bigr)  = \sum_{i=0}^{k}
(-1)^i a_0 \otimes \cdots \otimes a_i a_{i+1} \otimes \cdots \otimes a_{k+1}, \qquad k \geq 1. \]
Let $A^{op}$ be $A$ with the opposite product and  $A^e = A \otimes A^{op}$.
Then $B(A)$ is a complex of left $A^e$-modules by letting
\[ (a \otimes b) \cdot a_1 \otimes \cdots \otimes a_{k+2} = a \cdot a_1 \otimes
\cdots \otimes a_{k+2} \cdot b. \]
Any  $A\text{-}A$-bimodule $M$ can be viewed 
as a right $A^e$-module by letting $m \cdot (a \otimes b) = b
\cdot m \cdot a$. Then 
\[ B_\bullet(A,M) :=  M \otimes_{A^e} B(A) \]
is a complex of $\mb{F}$-vector spaces. The homology of this complex is known as the \emph{Hochschild homology} of $A$
with coefficients in $M$ and is denoted by $HH_\bullet(A,M)$. 

Given an $A\text{-}A$-bimodule
$M$, we obtain a complex of $\mb{F}$-vector spaces 
\[ C^\bullet(A,M) := \Hom_{A\text{-}A\text{-bimod}} \bigl(B_\bullet(A), M\bigr) . \]
The homology of this complex is known as the \emph{Hochschild cohomology} of $A$
with coefficients in $M$. It is easy to see that this cohomology coincides with
the one defined in Section \ref{ssec:Hoch}. 

For a unital algebra $A$, we will use the normalized Hochschild complex
$\bar{C}^\bullet(A,M)$ \cite[1.5.7]{L13} consisting on Hochschild cochains $f
\in C^\bullet(A,A)$ vanishing on elements of the form $a_0 \otimes \cdots
\otimes a_k$, where one of the $a_j$ is $1$. The inclusion $\bar{C}^\bullet(A,A)
\hookrightarrow C^\bullet(A,A)$ is a quasi-isomorphism. Indeed the map
\[ a_1 \otimes \cdots \otimes a_{k+2} \rightarrow 1 \otimes a_1 \otimes \cdots
\otimes a_{k+2}, \]
induces a homotopy between the identity map of $C^\bullet(A,A)$ and its
projection to $\bar{C}^\bullet(A,A)$. 
 Suppose that the algebra
$A$ is unital and augmented, with an augmentation ideal $A_+$; in this case we have 
\begin{equation}
\bar{C}^i(A,A) = \Hom_{\mb{F}}(A_+^{\otimes i}, A).
\label{eq:augmented-hh}
\end{equation}

\subsection{K\"ahler differentials}
Let $A$ be an associative commutative $\mb{F}$-algebra, and $I \subset A \otimes
A$ be the kernel of the multiplication map $A\otimes A \rightarrow
A$. The $A$-module $\Omega^1_A := I / I^2$ is called the \emph{module of
K\"ahler differentials} of $A$. 

For an $A$-module $M$, a \emph{derivation of $A$ with values in $M$} is a
linear map $D \in \Hom_{\mb{F}}(A, M)$ satisfying 
\[ D(a\cdot b) = a \cdot D(b) + b \cdot D(a). \]
The space of all derivations $\Der(A,M)$ is an $A$-module and we have
\[ \Der(A,M) \simeq \Hom_A(\Omega^1_A, M). \]
In particular, the identity map of $\Omega^1_A$ gives a derivation $d \in
\Der(A, \Omega^1_A)$; explicitly,
\[ d a = a \otimes 1 - 1 \otimes a  \mod I^2. \]
We define the module of $n$-forms by
\[ \Omega^n_A := \bigwedge\nolimits^n_A \Omega^1_A, \qquad n \geq 0. \]

Let $V$ be a $\mb{F}$-vector space, and consider the free
commutative associative unital algebra $A = S(V)$  generated by $V$. In this case,
\[ \Omega^n_{A} \simeq A \otimes \bigwedge\nolimits^n V, \]
since $\Omega^1_A$ is a free $A$-module of rank $=\dim V$.  We view
$\Omega^\bullet_A = \bigoplus_{n \geq 0} \Omega^n_A$ as a complex with zero
differential. 
We have the following map of complexes
$\varepsilon\colon \Omega^\bullet_A \rightarrow B_\bullet(A,A)$, called the
\emph{antisymmetrization map}, defined by 
\begin{equation} \label{eq:anti-symmetrization-hkr}
\varepsilon \bigl(a \otimes v_1 \wedge\cdots \wedge v_n \bigr) = \sum_{\sigma
\in S_n} \sign(\sigma) a \otimes
v_{\sigma^{-1}(1)} \otimes \cdots \otimes v_{\sigma^{-1}(n)}. 
\end{equation}

\begin{theorem}[{HKR Theorem \cite[Thm. 3.2.2]{L13}}] Let\/ $A = S(V)$ as above.
Then the antisymmetrization map $\varepsilon$, given by
\eqref{eq:anti-symmetrization-hkr},
is a quasi-isomorphism. In particular, we have an isomorphism $\varepsilon_* \colon
\Omega^n_A \xrightarrow{\sim}
HH_n(A,A)$ for all $n \geq 0$ induced in homology.  Its inverse is given by the 
surjective map 
\[ \pi_*\colon HH_\bullet(A,A) \rightarrow \Omega^\bullet_A, \qquad \pi_* \bigl(a_0
\otimes \cdots \otimes a_k \bigr)= a_0 da_1 \wedge \cdots \wedge da_k. \]
\label{thm:hkr}
\end{theorem}

Similarly, we have:
\begin{theorem}
For $A = S(V)$ as above, the inclusion of complexes
\[ \pi^\sharp\colon
\bigwedge\nolimits^\bullet_A \Der(A,A) \hookrightarrow C_\bullet(A,A)\] 
is a quasi-isomoprhism. Consequently, we have an isomorphism of cohomology
groups
\[ \varepsilon^\sharp_* \colon HH^\bullet(A,A) \rightarrow \bigwedge\nolimits^{\bullet}_A \Der(A,A) \simeq A
\otimes \bigwedge\nolimits^\bullet V^*, \]
defined as the inverse of the map $\pi^\sharp_*$ induced by $\pi^\sharp$ in
cohomology. 
\label{thm:dualhkr}
\end{theorem}

\subsection{The differential setting} \label{ssec:etingof}
Let now $A$ be a differential associative algebra, that is an associative
algebra over $\mb{F}$ with a derivation $\partial$. Then the complex $B(A)$ is a
complex of $\mb{F}[\partial]$-modules. Given a differential $A$-bimodule $M$, we
have the complex of $\mb{F}$-vector spaces
\[ C_{\partial}^\bullet(A,M) := \Hom^\partial_{A\text{-}A\text{-bimod}}(B_\bullet(A),M), \]
where the $\Hom$ is taken in the category of differential $A\text{-}A$-bimodules. The
homology of this complex is the \emph{differential Hochschild cohomology} of $A$
with coefficients in $M$, denoted by $HH^\bullet_{\partial}(A,M)$. It is
clear that this definition coincides with the definition in Section \ref{ssec:Hoch}. 

Let $A = \mb{F}[x_i^{(j)} \, | \, 1 \leq i \leq N, \,j \geq 0]$ be a differential polynomial
algebra in $N$ variables $x_i = x_i^{(0)}$ and their
derivatives $\partial x_i^{(j)} := x_i^{(j+1)}$, $j \geq 0$. Let $A_+ \subset A$
be the augmentation ideal. We will need the following well known result, whose proof we provide for
completeness
\begin{lemma}  $A_+$ is free as an\/
$\mb{F}[\partial]$-module. 
\label{lem:a+isfree}
\end{lemma}
\begin{proof}
Consider first the case when $A$ is
a differential polynomial algebra in one variable $x = x^{(0)}$, that is $A =
\mb{F}[x^{(0)},x^{(1)},\dots]$. 
An $\mb F$-basis of $A_+$ is
given by the monomials
\begin{equation}
 x^{(\lambda)} = x^{(\lambda_1)}\cdots x^{(\lambda_k)}, \qquad \lambda = \lambda_1 \geq \cdots \lambda_k \geq 0, \qquad k \geq 1. 
\label{eq:monomial-basis}
\end{equation}
We have
\begin{equation}
\partial x^{(\lambda)} = \sum_{i = 1}^k x^{(\lambda_1)}\cdots
x^{(\lambda_i+1)}\cdots
x^{(\lambda_k)}. 
\label{eq:partial-xs}
\end{equation}
The module $A_+$ is a graded $\mb F[\partial]$-module with  $\deg x^{(\lambda)}
= k + \sum_{i=1}^k \lambda_i$,
and $\deg \partial = 1$.  Notice that the homogeneous components $(A_+)_n$ of
degree $n$ are finite dimensional over $\mb F$. We consider the \emph{weighted reverse lexicographic
order} on the set of monomials \eqref{eq:monomial-basis}: for two
partitions $\lambda, \mu$ we let $x^{(\lambda)} > x^{(\mu)}$ if $\deg
x^{(\lambda)} > \deg x^{(\mu)}$ or $\deg x^{(\lambda)} = \deg x^{(\mu)}$ and
there exists $i_0 \geq 1$ such that $\lambda_i  = \mu_i$ for $1 \leq i \leq
i_0$ and $\lambda_{i_0} > \mu_{i_0}$.  This is a total ordering on the set of
monomials \eqref{eq:monomial-basis}. 

We construct an $\mb
F[\partial]$-basis of $A_+$ as follows. For each homogeneous degree component
$(A_+)_n$, we consider a set of monomials $\mathcal{B}_n \subset (A_+)_n$ such
that their images in $(A_+)_n \bigl\slash \partial (A_+)_{n-1}$ form an $\mb F$-basis.  
This set exists since we have a total ordering of a monomial basis of $(A_+)_n$
over $\mb F$. 
We let $\mathcal{B} = \coprod_{n \geq 1} \mathcal{B}_n$. We claim that
$\mathcal{B}$ is an $\mb F[\partial]$-basis of $A_+$. 

First, let us prove by induction that $\mathcal{B}$ spans $A_+$. Let $a \in A_+$ be a homogeneous element
of degree $n$. We prove by induction that $a$ can be written as a linear
combination with coefficients in $\mb F[\partial]$ of elements of $\mathcal{B}$.
When $n = 1$ there is nothing to prove as $(A_+)_1$ has as a basis $\mathcal{B}_1
= \{ x^{(0)} \}$. Assume that every homogeneous element of degree less than $n$
is in the $\mb F[\partial]$-span of $\mathcal{B}$. We can assume that $a$
is a monomial. By the definition of $\mathcal{B}_n$, there exist $b \in (A_+)_n$ and
$c \in (A_+)_{n-1}$ such that $a = b + \partial c$ with the property that $b$ is
an $\mb F$-linear combination of elements of $\mathcal{B}_n$. By our induction
hypothesis, $c$ (and therefore $\partial c$) can be written as an $\mb
F[\partial]$-linear combination of elements of
$\mathcal{B}$. Therefore, $\mathcal{B}$ spans $A_+$ over $\mb F[\partial]$. 

Let us now prove that the
elements of $\mathcal{B}$ are linearly independent over $\mb F[\partial]$.
Suppose we are given $b_1, \dots, b_r \in \mathcal{B}$ such that 
\begin{equation} \label{eq:ordering-linear-comb}
 \alpha_1 \partial^{j_1} b_1 + \cdots + \alpha_r \partial^{j_r} b_r = 0, \qquad \alpha_i \in
\mb F, \quad \alpha_i \neq
0, \quad j_i \geq 0. 
\end{equation}
We may assume that each summand is homogeneous of degree $n$ and that $j_1 \geq
\cdots \geq j_r$.  Since $\partial$ is injective on $A_+$, we may assume that
$j_r = 0$. Let $i_0$ be the minimum such that $j_{i_0} = 0$. Thus $b_i \in
\mathcal{B}_n$ for $i_0 \leq i \leq r$.  
It follows that 
$\sum_{i= i_0}^r \alpha_i b_i$ vanishes modulo $\partial (A_+)_{n-1}$, which contradicts our
choice of $\mathcal{B}_n$.
This proves that
$\mathcal{B}$ is an $\mb F[\partial]$-basis of $A_+$. 

For general $N$, writing $A^N$ in place of $A$ and denoting the case $N=1$ again by $A$, we have
an isomorphism of $\mb F[\partial]$-modules
\[A^N \simeq A^{\otimes N} = (A_+ \oplus \mb F 1)^{\otimes N}.\]
Hence, the augmentation ideal $(A^N)_+$ is a direct sum of 
tensor products of free $\mb{F}[\partial]$-modules, and so is free. 
\end{proof}

Now we introduce the subspace of poly-vector fields
\[ P^\bullet \subset \Hom_{\mb{F}} (A^{\otimes \bullet}, A), \]
i.e., alternating maps that are 
derivations in each argument. We consider $P^\bullet$ as a complex with the zero
differential. Since $A$ is a differential algebra, $P^\bullet$ is naturally an
$\mb{F}[\partial]$-module; let $P^\bullet_{\partial} = \ker \partial$.

\begin{theorem} Let\/ $A = \mb{F}[x_i^{(j)} | \, 1 \leq i \leq N, \, j
\geq 0]$ be a differential polynomial
algebra in $N$ variables and their derivatives. 
Then for all $k \geq 0$ we have an isomorphism 
\[ HH^k_{\partial}(A,A) \simeq P^k_\partial. \]
\label{thm:pasha}
\end{theorem}
\def\Ext{\mathrm{Ext}}
\begin{proof}(P. Etingof)
We consider the normalized Hochschild complex $\bar{C}^\bullet(A,A)$ defined in
\eqref{eq:augmented-hh}. 
It follows from Theorem \ref{thm:dualhkr} that $\pi^\sharp\colon P^\bullet \hookrightarrow
\bar{C}^\bullet(A,A)$ is a quasi-isomorphism. Notice also that the inclusion
$\pi^\sharp$ commutes with the $\mb{F}[\partial]$-action. That is, the complexes
$P^\bullet$ and $\bar{C}^\bullet(A,A)$ are quasi-isomorphic as complexes of
$\mb F[\partial]$-modules. Considering $\mathbb{F}$ as a trivial
$\mb F[\partial]$-module, it follows that we have a quasi-isomorphism of complexes
of vector spaces:
\begin{equation}
 R\Hom_{\mb{F}[\partial]}(\mb{F}, P^\bullet) \rightarrow
R\Hom_{\mb{F}[\partial]}(\mb{F}, \bar{C}^\bullet(A,A)), 
\label{eq:etingof1}
\end{equation} 
where $R\Hom$ is the right derived functor of $\Hom$, whose cohomology computes the $\Ext$ groups. 
To compute the cohomology of these complexes, we
consider the resolution 
\begin{equation} \label{eq:koszresol1} \mb{F}[\partial] \xrightarrow{\partial \cdot} \mb{F}[\partial]
\twoheadrightarrow \mb{F}. 
\end{equation}
We replace $\mb F$ by the two term complex $\mb F[\partial] \rightarrow \mb
F[\partial]$ in \eqref{eq:etingof1} and therefore the space of morphisms of $\mb
F[\partial]$-modules 
\[ \Bigl( \mb F[\partial] \xrightarrow{\partial \cdot} \mb F[\partial] \Bigr)
\rightarrow P^\bullet,  \qquad 
 \Bigl( \mb F[\partial] \xrightarrow{\partial \cdot} \mb F[\partial] \Bigr)
\rightarrow \bar{C}^\bullet(A,A),
\]
are naturally  bi-complexes of vector spaces. They consist of complexes with two
rows and infinitely many columns. Thus, the cohomology of the complexes in \eqref{eq:etingof1}
are given by the cohomology of the total complexes associated to the two-row
bicomplexes $P^\bullet \xrightarrow{\partial} P^\bullet$ and
$\bar{C}^\bullet(A,A) \xrightarrow{\partial} \bar{C}^\bullet(A,A)$. We compute
the vertical cohomology of the complex $P^\bullet \xrightarrow{\partial} P^\bullet$ first.

We claim that the map $\partial\colon P^i \rightarrow P^ i$ is surjective for $i \geq
1$. In fact, if we let $T^i \subset \bar{C}(A,A)^i$ be the subspace of all maps
that are derivations on each argument, we see that $P^i \hookrightarrow T^i$ is
a split injection since $T^i$ decomposes as a representation of the symmetric
group $S_i$ on $i$ elements. It suffices to prove that $\partial\colon T^i
\rightarrow T^i$ is surjective for $i \geq 1$. This is equivalent to showing
that
$\Ext^1_{\mb{F}[\partial]}(\mb{F}, T^i)=0$ for $i \geq 1$. Indeed, we may replace
$\mb F$ by \eqref{eq:koszresol1} and computing the $\Ext$ groups amounts to
computing the cohomology of the complex $T^i \xrightarrow{\partial} T^i$, which
vanishes in degree $1$ if and only if $\partial$ is surjective.
 Notice that 
\[ T^i = \Hom_{\mb{F}} \bigl( (A_+/A_+^2)^{\otimes i}, A \bigr),\] since a
derivation is determined on $A_+^2$ by the Leibniz rule. Also note that $A_+ /
A_+^2$ is a free $\mb{F}[\partial]$-module $M \simeq \mb{F}[\partial]^N$, with basis given by $\{ x_i^{(0)}
\}_{1 \leq i \leq N}$. Since $M$ is a free $\mb{F}[\partial]$-module, we obtain
\[ \Ext^1_{\mb{F}[\partial]}(\mb{F},T^i) = \Ext^1_{\mb{F}[\partial]}\left(\mb{F},
\Hom_{\mb{F}}\left(  M^{\otimes i}, A \right) \right)  = 0, \qquad i \geq 1. \]
Since the horizontal differentials of $P^\bullet \xrightarrow{\partial}
P^\bullet$ vanish (as the differential of $P^\bullet$ vanishes), we obtain that the
total cohomology of the bicomplex $P^\bullet\xrightarrow{\partial} P^\bullet$ is
given as follows. In degree $i \geq 2$, it is $P^i_{\partial}$, that is the
$\varphi^i \in P^i$ such that $\partial \varphi^i =0$. In degree $1$, we have
$P^1_\partial \oplus A / \partial A$, the first summand corresponds to the
vertical cohomology in degree $0$ of $P^1$ while the second is the vertical
cohomology of degree $1$ of $P^0 = A$. Finally, in degree $0$, we have
$P^0_{\partial} = \mb{F}$. 

We now consider the cohomology of the complex $\bar{C}^\bullet(A,A)
\xrightarrow{\partial} \bar{C}^\bullet(A,A)$ which computes the right-hand side of
\eqref{eq:etingof1}. 
It follows from Lemma \ref{lem:a+isfree} that 
\[ \Ext^1_{\mb{F}[\partial]}(\mb{F},\bar{C}_i(A,A)) =
\Ext^1_{\mb{F}[\partial]}(\mb{F}, \Hom_{\mb{F}}(A^{\otimes i}_+, A)) =
\Ext^1_{\mb{F}[\partial]}(A_+^{\otimes i}, A) = 0, \qquad i \geq 1. \]
Thus, the vertical differentials of $\bar{C}^\bullet(A,A)
\xrightarrow{\partial} \bar{C}^\bullet(A,A)$ are also surjective for $i \geq 1$.
The vertical cohomology of this bicomplex is therefore $\bar{C}^i_\partial(A,A)$
for $i \geq 1$, while in the first column we have the cohomology
$\bar{C}_\partial^0(A,A) = A^\partial = \mb{F}$ in degree $0$ and
$\bar{C}^0(A,A) / \partial \bar{C}^0(A,A) = A / \partial A$ in degree $1$.
Computing now the horizontal cohomology, we obtain that the total cohomology of
the bicomplex $\bar{C}^\bullet(A,A) \rightarrow \bar{C}^\bullet(A,A)$ consists
of $H^i(\bar{C}_\partial(A,A))$ for $i \geq 2$. In degree $1$ we have
$H^1(\bar{C}_\partial(A,A)) \oplus A/\partial A$, and in degree $0$ we have
$\mb{F}$. We have therefore obtained $H^i(\bar{C}_\partial(A,A)) \simeq
P^i_\partial$ for all $i\geq 0$ as claimed. 
\end{proof}
\subsection{The sesquilinear setting}
Let $A$ be an associative differential algebra and $s \geq 1$. Consider the
total complex of the $s$-complex
\[ B(A)^{\otimes s} = \underbrace{B(A) \otimes_A \cdots \otimes_A B(A)}_{s \text{ times}}. \]
This is a complex of $A$-$A$-bimodules and of $\mb{F}[\partial_1,\dots,\partial_s]$-modules. Let $M$ be
a differential $A$-$A$-bimodule. Define 
\[ \Delta^s M = M \otimes_{\mb{F}[\partial]} \mb{F} [\partial_1, \dots,
\partial_s], \]
where the left $\mb{F}[\partial]$-module structure on $\mb{F}[\partial_1,
\dots, \partial_s]$ is given by the diagonal map $\partial \mapsto \sum
\partial_i$.  Then $\Delta^s M$ is an $A$-$A$-bimodule and an
$\mb F[\partial_1,\dots,\partial_s]$-module. 
We have the complexes 
\[ C^{s, \bullet} = \Hom_{A\text{--}A \text{--bimod}}( B(A)^{\otimes s}, \Delta^s M), \qquad C^{s,
\bullet}_{\partial} = \Hom( B(A)^{\otimes s}, \Delta^s M),  \]
the $\Hom$ in the right-hand side being taken in the category of $A$-$A$-bimodules and
$\mb{F}[\partial_1,\dots,\partial_s]$-modules. It is clear from the definition
that $C^{s,\bullet}_{\partial}(A,A)$ coincides with the complex
$C^{s,\bullet}_{\Hoc}$ from \eqref{20200625:eq2}. 

\begin{remark} \label{rem:symmetricgroupaction}
Notice that the complexes $C^{s,\bullet}$ and $C^{s,\bullet}_\partial$ decompose
under the action of products of symmetric groups as follows. For each degree
$i$ and a partition $k_1 + \cdots + k_s = i$, the complex $C^{s,\bullet}$ has a direct summand consisting of maps 
\[ B^{k_1}(A) \otimes \cdots \otimes B^{k_s}(A) \rightarrow \Delta^s M. \]
The group $S_{k_1} \times \cdots \times S_{k_s}$ acts by permuting the entries on the left-hand side. The complex $C^{s, \bullet}$ is a direct sum of these symmetric group representations for all $i$ and all partitions. 
\end{remark}

Let now $A$ be in addition commutative,
let $\Omega^1_A$ be the module of K\"ahler differentials of $A$, and let 
\[ \Omega^\bullet_A = \bigwedge\nolimits^\bullet_A \Omega^1_A \]
be the module of differential forms. We consider $\Omega^\bullet_A$ as a complex with zero
differential. We have $P^\bullet = \Hom(\Omega^\bullet, A)$. 
Note that $\Omega^1_A$ and
therefore $\Omega^\bullet_A$ are differential $A$-modules. 
Hence
$P^\bullet_\partial = \Hom_{A-\mb{F}[\partial]}(\Omega^\bullet_A, A)$, where the $\Hom$ is taken in the category of differential $A$-modules. 

\begin{theorem} \label{thm:sesquihkr}
Let\/ $A = \mb{F}[x_i^{(j)}\,|\, 1 \leq i \leq N,\, j \geq 0]$ be a differential
polynomial algebra, and\/ $M$ be its differential module.  Then for every $i \geq 0$ we have isomorphisms
\[ 
\begin{aligned}
H^i( C^{s,\bullet}(A,M) ) &\simeq H^i \left( \Hom( (\Omega^\bullet_A)^{\otimes s}, \Delta^s M )\right), \\ 
H^i( C^{s,\bullet}_\partial(A,M) ) &\simeq H^i\left(  \Hom_{A-\mb{F}[\partial_1,\ldots,\partial_s]} ( (\Omega^\bullet)^{\otimes s}, \Delta^s M) \right).
\end{aligned}
\] 
\end{theorem}
\begin{proof}
The first isomorphism is simply a consequence of the HKR Theorem \ref{thm:hkr} for $A$ stating
that $B(A)$ is quasi-isomorphic to $\Omega^\bullet_A$. Since the latter is a
free $A$-module, it is flat, and therefore $B(A)^{\otimes s}$ is quasi-isomorphic to $(\Omega^\bullet_A)^{\otimes s}$. The result follows by taking Homs into $\Delta^s M$. 

The quasi-isomorphism $B(A)^{\otimes s} \rightarrow (\Omega^\bullet_A)^{\otimes s}$ is a quasi-isomorphism of complexes of $A$-modules and $\mb F[\partial_1,\dots,\partial_s]$-modules. It follows that we have a quasi-isomorphism of complexes of $A$-modules and $\mb {F}[\partial_1,\dots,\partial_s]$-modules
\[ C^{s,\bullet}(A,M) \rightarrow \Hom_A \left( (\Omega_A^\bullet)^{\otimes s}, \Delta^s M \right), \]
and hence the following two complexes are quasi-isomorphic
\begin{equation} \label{eq:quasi-etingof-2}
 R\Hom_{\mb{F}[\partial_1,\ldots,\partial_s]} \bigl( \mb{F}, C^{s,\bullet}(A,M) \bigr) \rightarrow R\Hom_{\mb{F}[\partial_1,\ldots,\partial_s]} \bigl( \mb{F},\Hom_A \left( (\Omega_A^\bullet)^{\otimes s}, \Delta^s M \right)\bigr) .
\end{equation}
\def\Omfpar{\Omega_{\mb{F}[\partial_1,\ldots,\partial_s]}^1}
In order to compute the cohomology of \eqref{eq:quasi-etingof-2}, we use the
Koszul resolution of $\mb{F}$ as an
$\mb{F}[\partial_1,\dots,\partial_s]$-module. We consider the free module
$\Omfpar$ with a basis $d^1,\dots,d^s$ and the resolution 
\begin{equation} \label{eq:koszresol2} \cdots \rightarrow
\bigwedge\nolimits^k \Omfpar \rightarrow \bigwedge\nolimits^{k-1} \Omfpar \rightarrow \cdots \rightarrow \Omfpar
\rightarrow \mb{F}[\partial_1,\ldots,\partial_s] \rightarrow \mb {F}. 
\end{equation}
This resolution coincides with the two-term resolution \eqref{eq:koszresol1}
when $s=1$. 

The complex \eqref{eq:koszresol2} is non-negatively graded, with $\mb{F}[\partial_1,\dots,\partial_s]$ in degree 0. 
The Koszul differential is defined by $d^i \mapsto \partial^i$ and extending by
the Leibniz rule to a derivation of degree $-1$ of the free commutative superalgebra $\bigwedge\nolimits^\bullet \Omfpar$. 
Hence, in order to compute the cohomology of \eqref{eq:quasi-etingof-2}, we need to compute the cohomology of the total complexes with $s+1$ rows 
\begin{equation} \label{eq:twokoszulcomplexes} \bigwedge\nolimits^{\bullet} \mathbb{F}^s \otimes C^{s,\bullet}(A,M) \quad \text{and} \quad 
\bigwedge\nolimits^{\bullet} \mb{F}^s \otimes \Hom_A\left( (\Omega_A^\bullet)^{\otimes s},\Delta^s M \right). 
\end{equation}
We compute first the vertical cohomology of the complex on the right. We claim that for each column $i \geq 1$ the vertical cohomology in 
$\bigwedge\nolimits^{\bullet} \mb{F}^s \otimes \Hom\left( (\Omega_A^\bullet)^{\otimes
s},\Delta^s M \right)$ vanishes in positive degrees. Indeed, let $T^{s,\bullet}$
be the set of all maps in $C^{s,\bullet}(A,M)$ that are derivations in each
argument. We have a split injection $\Hom( (\Omega_A^\bullet)^{\otimes s},
\Delta^s M )$ $\hookrightarrow T^i$, since the latter splits as a representation of
the symmetric group as in Remark \ref{rem:symmetricgroupaction}. It suffices
then to prove that the vertical cohomology of $\bigwedge\nolimits^{\bullet} \mb{F}^s
\otimes T^{s,\bullet}$ vanishes in positive degrees. For each partition $k_1 + \cdots + k_s = i \geq 1$, the corresponding summand of $T^{s,\bullet}$ is given by maps 
\begin{equation} \label{eq:pasha3}
 (A_+/A_+^2)^{\otimes k_1} \otimes \cdots \otimes (A_+/A_+^2)^{\otimes k_s} \rightarrow \Delta^s M. 
\end{equation}
Notice that if some $k_i = 0$, the corresponding space of maps vanishes since there are no non-trivial derivations of $\mb{F}$. So we may assume that all $k_i > 0$. 
Since $A_+/A_+^2$ is free as an $\mb F[\partial]$-module, it follows that the left-hand side of \eqref{eq:pasha3} is free as an $\mb F[\partial_1,\dots,\partial_s]$-module. Hence 
\[ \Ext^j_{\mb F[\partial_1,\ldots,\partial_s]} (\mb F, T^i) = 0, \qquad i,j \geq 1, \] 
proving that the vertical cohomology of the second complex in
\eqref{eq:twokoszulcomplexes} vanishes in positive degrees for each column $i
\geq 1$. The zeroth column is given by the complex $\bigwedge\nolimits^{\bullet} \mb{F}
\otimes \Delta^s M$, where the differential is defined by $d^i \otimes m \mapsto
m \partial_i$ extended to a derivation of degree $-1$.  Since the horizontal differentials are zero, we obtain the following description of the total cohomology. In each degree $i \geq 1$, we have 
\begin{equation} H^i \left( \Hom_{A-\mb{F}[\partial_1,\ldots, \partial_s]}\left(
(\Omega^\bullet)^{\otimes s}, \Delta^s M \right) \right) \oplus H^{i} \left( \bigwedge\nolimits^{\bullet} \mb{F}^s \otimes \Delta^s M \right),
\label{eq:zerothcolsummand}
\end{equation}
where the first summand corresponds to the $i$-th horizontal cohomology of the zeroth row, while the second is the $i$-th vertical cohomology of the zeroth column. In degree $0$, we have $\mathbb{F}$. 

We now analyze the vertical cohomology of the first bicomplex in
\eqref{eq:twokoszulcomplexes}. We notice that, in the same way as in the proof
of Theorem \ref{thm:pasha},  for any partition $k_1 + \cdots + k_s = i$ where all $k_i > 0$, we have $(A_+)^{\otimes \sum k_j}$ is a free $\mb{F}[\partial_1,\dots,\partial_s]$-module. Hence, we obtain 
\[ \Ext^j_{\mb {F}[\partial_1, \ldots, \partial_s]} \left( \mb F, C^{s, j} \right) = 0, \qquad i,j \geq 1. \]
The cohomology is therefore again concentrated in the zeroth row and the zeroth
column. The zeroth vertical cohomology is given by $C^{s, \bullet}_\partial$,
while the zeroth column is given by $\bigwedge\nolimits^{\bullet} \otimes \Delta^s M$. We
see that the zeroth column contributes the same cohomology to the second summand
of \eqref{eq:zerothcolsummand}, while the horizontal cohomology of the zeroth
row is now given by $H^i\left(C^{s, \bullet}_{\partial}(A,M)\right)$, proving the theorem. 
\end{proof}
\subsection{The sesquilinear Hodge decomposition}
We recall here the \emph{Hodge decomposition} of the Hochschild cohomology of a commutative algebra $A$ with coefficients in its module $M$; see \cite{GS87,L13}. 
The symmetric group $S_n$ acts on
\[ C^n := C^n(A,M) \simeq \Hom (A^{\otimes n}, M) \]
by permuting the $n$ factors.
Recall the \emph{Eulerian idempotents} $e^{(i)}_n \in \mb Q[S_n]$ of the group algebra of $S_n$ (see \cite[4.5.2]{L13} for an explicit description). They satisfy
\[ 
\begin{gathered}
1 = e_n^{(1)} + \cdots + e^{(n)}_n,  \\ 
e^{(i)}_n e^{(j)}_n = 0, \quad \text{if} \quad i \neq j, \qquad \text{and} \quad e^{(i)}_n e^{(i)}_n = e^{(i)}_n. 
\end{gathered}
\]
It follows from \cite[4.5.10]{L13} that, putting $C^n_{(k)} := e_n^{(k)} C^n$,
and letting $HH^n_{(k)}(A,M) \subset HH^n(A,M)$ consist of cohomology classes of
elements in $C^n_{(k)}$, we obtain a direct sum decomposition 
\[ HH^n(A,M) = HH^n_{(1)}(A,M) \oplus \cdots \oplus HH^n_{(n)}(A,M), \qquad n
\geq 1.\]
The first summand $HH^n_{(1)}(A,M)$ is identified canonically with the Harrison cohomology $H^n(C^\bullet_{\Har}(A,M))$ 
by \cite[4.5.13]{L13}. The last summand is identified with polyvector fields \cite[4.5.13]{L13}:
\begin{equation} \label{eq:hodge-poly} HH^n_{(n)}(A,M) \simeq \Hom \left( \bigwedge\nolimits^n \Omega^1_A, M \right). 
\end{equation}

The above description generalizes to the sesquilinear setting. Recall from the
proof of Proposition \ref{20200625:prop2} that the complexes
$C^{s,\bullet}_{\partial}(A,M)$ are complexes in the category of representations
of symmetric group $S_s$, so that the action of the symmetric group $S_s$
as described in Section \ref{ssec:sesquilinear-hh} commutes with the differential. In
addition, it preserves the Harrison conditions \eqref{20200623:eq3}. For each
$s$ and $k_1 + \cdots + k_s = n$, we have an action of the product $S_{k_1}
\times \cdots \times S_{k_s}$ on $C^{s, \bullet}_{\partial}(A,M)$ by permuting
the entries and commuting with the differential. Consider the corresponding
\emph{Eulerian idempotents} $e^{(i)}_{k_j} \in \mb Q[S_{k_j}]$, for $i \geq 0$ and $j=1,\cdots,s$. 
For $\underline{i} = (i_1,\cdots,i_s)$ and $\underline{k} = (k_1,\cdots,k_s)$, we let 
\[ e^{(\underline{i})}_{\underline{k}} = e^{(i_1)}_{k_1} \otimes \cdots \otimes e^{(i_s)}_{k_s} \in \mb Q[S_{k_1} \times \cdots \times S_{k_s}]. \]
For each $\underline{i}$, we set
\[ C^{s, n}_{(\underline{i}), \partial} (A,M) = \bigoplus_{k_1 + \cdots + k_s = n}
e^{(\underline{i})}_{\underline k} C^{s, \underline{k}}_{\partial} (A,M).
\]
We obtain the corresponding decomposition of the sesquilinear Hochschild cohomology
\[H^n \bigl( C^{s,\bullet}_{\partial}(A,M) \bigr) = \bigoplus_{\underline{i}}
H^n \bigl( C^{s,\bullet}_{(\underline i), \partial}(A,M) \bigr). \]
Denote by $\underline{1}$ the $s$-tuple $(1,\dots,1)$. The summand
for $\underline{i} = \underline{1}$ is identified with the
sesquilinear Harrison cohomology in the same way as above:
\begin{equation} \label{eq:sesqui-harrison-hodge}
 H^n\bigl(C^{s,\bullet}_{\sesq,\Har} (A,M) \bigr) = H^n \bigl(
C^{s,\bullet}_{(\underline{1}), \partial}(A,M) \bigr). 
\end{equation}
In the other extreme case, we obtain from \eqref{eq:hodge-poly} the
identification of sesquilinear polyvector fields with the following sum
\begin{equation} \label{eq:sesqui-poly} 
H^n\left(  \Hom_{A-\mb{F}[\partial_1,\ldots,\partial_s]} (
(\Omega^\bullet)^{\otimes s}, \Delta^s M) \right) \simeq
\bigoplus_{k_1 + \dots + k_s=n} H^n \bigl( C^{s,\bullet}_{(\underline k),
\partial} (A, M) \bigr).
\end{equation}
\label{ssec:sesquilinear-hodge}

The main result of this section is the following:

\begin{theorem}\label{thm:reimundo}
Let\/ $A = \mb{F}[x_i^{(j)} \, | \, 1 \leq i \leq N, \, j \geq 0]$ be a differential
polynomial algebra, and\/ $M$ be its differential module. Then for every $n > s>0$
the sesquilinear Harrison cohomology of\/ $A$ with coefficients in $M$ vanishes:
\[  H^n\bigl(C^{s,\bullet}_{\sesq,\Har} (A,M) \bigr)  = 0. \]
\end{theorem}
\begin{proof}
Let $n > s>0$ and consider the sesquilinear Hochschild cohomology of $A$ with coefficients in $M$,
namely $H^n( C^{s, \bullet}_{\partial}(A,M))$. 
By Theorem \ref{thm:sesquihkr} and \eqref{eq:sesqui-poly}, we have an isomorphism 
\[ H^n( C^{s,\bullet}_\partial(A,M) ) \simeq 
\bigoplus_{k_1 + \dots + k_s=n} H^n \bigl( C^{s,\bullet}_{(\underline k),
\partial} (A, M) \bigr).
\]
If $n > s$ this implies that in the sum in the right-hand side we must have some $k_i >1$,
and hence $\underline k \neq \underline 1$. This implies that
\[ H^n\big(C_{(\underline 1),\partial}^{s,\bullet}(A,M) \bigr) = 0, \]
and therefore the theorem follows by \eqref{eq:sesqui-harrison-hodge}.
\end{proof}
\begin{corollary} \label{cor:reimundo}
With the notation of Theorem \ref{thm:reimundo}, for every $n > s > 0$ the
symmetric $s$-sesquilinear Harrison cohomology of $A$ with coefficients in $M$ vanishes:
\[  H^n\bigl(C^{s,\bullet}_{\sym,\Har} (A,M) \bigr)  = 0. \]
\end{corollary}
\begin{proof}
It follows from Proposition \ref{20200625:prop2} that the sesquilinear Harrison
cohomology complex
$C^{s,\bullet}_{\sesq,\Har}$ is a complex of $S_s$-modules. The symmetric
$s$-sesquilinear
Harrison cohomology complex $C^{s,\bullet}_{\sym, \Har}(A,M)$ is defined in
\eqref{20200625:eq5} as its subcomplex of $S_s$-invariants. It follows that the
symmetric $s$-sesquilinear Harrison cohomology is a direct summand of the
$s$-sesquilinear Harrison cohomology. 
\end{proof}
\section{Proof of the Main Theorem \ref{thm:main}}\label{sec:8}

Recall that by Theorem \ref{thm:homomorphism}
the map \eqref{eq:homomorphism} is injective,
and we only need to prove that it is surjective.
In other words, we need to show that
for every closed element $Y\in C^n_{\cl}$ in the classical complex, $dY=0$,
there exist $Z\in C^{n-1}_{\cl}$ and $\widetilde Y\in C^n_{\cl}$ such that
\begin{equation}\label{20200910:eq1}
Y=dZ+\widetilde Y
\,,
\end{equation}
and
\begin{equation}\label{20200910:eq2}
\widetilde Y^\Gamma=0
\,\,\text{ if }\,\,
|E(\Gamma)|\neq0
\,.
\end{equation}
Recall the filtration $F_s C^n_{\cl}$ of the classical complex, given by equation \eqref{20200702:eq1}.
Clearly, $Y\in F_1 C^n_{\cl}=C^n_{\cl}$,
and the condition \eqref{20200910:eq2} on $\widetilde Y$ is equivalent to saying 
that $\widetilde Y\in F_n C^n_{\cl}$.
Hence, by induction, it suffices to prove that, 
for $1\leq s\leq n-1$ and $Y_s\in F_s C^n_{\cl}$ such that $dY_s=0$,
we can find 
$Z_s\in C^{n-1}_{\cl}$ and $Y_{s+1}\in F_{s+1} C^n_{\cl}$ 
satisfying
\begin{equation}\label{20200910:eq3}
Y_s=dZ_s+Y_{s+1}
\,.
\end{equation}
Consider the coset $Y_s+F_{s+1}C^n_{\cl}\in \gr_s C^n_{\cl}$.
Then, since the differential $d$ of $C_{\cl}$ preserves the filtration \eqref{20200702:eq1},
$Y_s+F_{s+1}C^n_{\cl}$ is a closed element of the complex $\gr_s C_{\cl}$.
By Theorem \ref{thm:whatever},
the complex $\gr_s C_{\cl}$ is isomorphic to the complex $C^s_{\sym,\Har}$,
which, by Corollary \ref{cor:reimundo}, has trivial $n$-th cohomology, since $s\leq n-1$.
As a consequence,
there exists $Z_s+F_{s+1}C^{n-1}_{\cl}\in \gr_s C^{n-1}_{\cl}$
such that 
$$
Y_s+F_{s+1}C^n_{\cl}
=
d(Z_s+F_{s+1}C^{n-1}_{\cl})
\,.
$$
This is equivalent to $Y_{s+1}:=Y_s-dZ_s\in F_{s+1}C^n_{\cl}$, proving the theorem.



\end{document}